\documentclass{amsart}
\usepackage[utf8]{inputenc}
\usepackage[polutonikogreek,greek,japanese,french,english]{babel}
\usepackage[scr=kp]{mathalpha}
\usepackage{dutchcal}
\usepackage[dvipsnames,svgnames,table]{xcolor}
\usepackage{mathtools}
\usepackage{amsmath}
\usepackage{amsthm}
\usepackage{geometry}
\usepackage{amssymb}
\usepackage{bbold}
\usepackage[all]{xy}
\usepackage{etoolbox}
\usepackage[hidelinks]{hyperref}
\usepackage{enumitem}
\usepackage[pdf]{pstricks}
\usepackage{pgfplots}
\pgfplotsset{compat = newest}
\usepackage{csquotes,xpatch}

\theoremstyle{plain}% default
\newtheorem{tho}[subsubsection]{Theorem}
\newtheorem{thintro}[subsection]{Theorem}
\newtheorem{lemme}[subsubsection]{Lemma}
\newtheorem{prop}[subsubsection]{Proposition}
\newtheorem{propintro}[subsection]{Proposition}
\newtheorem{cor}[subsubsection]{Corollary}

\theoremstyle{remark}
\newtheorem{rmq}[subsubsection]{Remark}

\theoremstyle{definition}
\newtheorem{deff}[subsubsection]{Definition}
\newtheorem{defintro}[subsection]{Definition}
\newtheorem{para}[subsubsection]{}

\newcommand{\cC}{\mathcal{C}}

\newcommand{\cE}{\mathcal{E}}
\newcommand{\cF}{\mathcal{F}}

\newcommand{\cH}{\mathcal{H}}
\newcommand{\cI}{\mathcal{I}}

\newcommand{\cL}{\mathcal{L}}

\newcommand{\cM}{\mathscr{M}}
\newcommand{\cN}{\mathscr{N}}
\newcommand{\cO}{\mathcal{O}}

\newcommand{\cQ}{\mathcal{Q}}

\newcommand{\cS}{\mathcal{S}}
\newcommand{\cT}{\mathcal{T}}

\newcommand{\cV}{\mathcal{V}}
\newcommand{\cW}{\mathcal{W}}

\newcommand{\eE}{\mathbb{E}}
\newcommand{\eF}{\mathbb{F}}

\newcommand{\eN}{\mathbb{N}}

\newcommand{\eP}{\mathbb{P}}

\newcommand{\eS}{\mathbb{S}}

\newcommand{\eV}{\mathbb{V}}

\newcommand{\eZ}{\mathbb{Z}}
\newcommand{\eUn}{\mathbb{1}}

\newcommand{\fD}{\mathfrak{D}}
\newcommand{\fF}{\mathfrak{F}}

\newcommand{\fH}{\mathfrak{H}}

\newcommand{\fp}{\mathfrak{p}}
\newcommand{\fq}{\mathfrak{q}}

\newcommand{\fU}{\mathfrak{U}}

\newcommand{\bfQ}{\mathbf{Q}}

\newcommand{\bfW}{\mathbf{W}}

\newdir{ >}{{}*!/-8pt/@{>}}

\newcommand{\Hom}[2][]{\mathrm{Hom}_{#1}\left(#2\right)}

\newcommand{\Thom}[2][]{\operatorname{Th}_{#1}\left(#2\right)}

\newcommand{\Maps}[2][]{\operatorname{Maps}_{#1}\left(#2\right)}

\newcommand{\GW}{GW}
\newcommand{\fId}{{\mathbf{I}}}
\newcommand{\Witt}{\operatorname{W}}
\newcommand{\KMW}[2][\ast]{K^{MW}_{#1}\!\left( #2 \right)}
\newcommand{\sKW}{\underline{K}^W}
\newcommand{\RostS}[4][]{\cC^{#1}\left(#2,\underline{K}^{MW}_{#3}\left(#4\right)\right)}
\newcommand{\RostSO}[3][]{\cC^{#1}\left(#2,\underline{K}^{MW}_{#3}\right)}
\newcommand{\RSgrp}[4]{H^{#1}\left(#2,\underline{K}_{#3}^{MW}\left(#4\right)\right)}
\newcommand{\RSgrpO}[3]{H^{#1}\left(#2,\underline{K}_{#3}^{MW}\right)}
\newcommand{\Chowtilde}[2][]{\widetilde{\operatorname{CH}}{}^{#1}\!\left(#2\right)}
\newcommand{\Chowtildecov}[2][]{\widetilde{\operatorname{CH}}_{#1}\left(#2\right)}
\newcommand{\StableHo}[2][]{\operatorname{S}\!\cH^{#1}\!\left(#2\right)}

\newcommand{\fvirt}[1]{\left\langle #1\right\rangle}

\newcommand{\Sq}{\operatorname{Sq}}
\newcommand{\simquad}{\sim_q}

\newcommand{\Picard}[1]{\mathrm{Pic}\left(#1\right)}

\newcommand{\img}{\mathrm{Im}}

\newcommand{\qscal}[1]{\left\langle #1 \right\rangle}
\newcommand{\qqscal}[1]{\left\langle\left\langle #1 \right\rangle\right\rangle}
\newcommand{\Id}{\mathrm{Id}}

\newcommand{\SL}{\operatorname{SL}}
\newcommand{\Gl}{\operatorname{GL}}
\newcommand{\isomr}{{\xrightarrow{\sim}}}

\newcommand{\Spec}[1]{\operatorname{Spec}\left(#1\right)}

\newcommand{\Proj}[1]{\operatorname{Proj}\left(#1\right)}
\newcommand{\Symalg}[2][]{\operatorname{Sym}^{#1}\left(#2\right)}

\newcommand{\Grassm}[2][]{\operatorname{Grass}_{#1}\left(#2\right)}

\newcommand{\Grass}{\operatorname{Gr}}

\newcommand{\Hilbsch}[2][]{\mathcal{H}ilb_{#1}\left(#2\right)}
\newcommand{\Hdrtes}{\Sigma_1}

\newcommand{\Chowct}[2][]{\operatorname{CH}^{#1}\!\left( #2 \right)}

\newcommand{\Chd}[2][]{\operatorname{Ch}^{#1}\!\left( #2 \right)}

\newcommand{\rest}[2]{{#1}_{\mid #2}}

\newcommand{\accol}[1]{\left\{#1\right\}}

\title{A quadratically enriched count of lines in smooth del Pezzo surfaces of degree 2 and 4}
\author{Victor CHACHAY}
\date{\today}

\begin{document}
\maketitle

\begin{abstract}
We give a computation of some Euler classes in Chow-Witt groups associated to the count of lines of smooth del Pezzo surfaces of degree 2 and 4. The description of Chow-Witt groups of projective bundles over Grassmannians for vector bundles that are not relatively orientable is the main part of the article. We show that the quadratic count is not enriched as the Chow-Witt group is isomorphic to the Chow group. In this setting, we give an expression of the classes of even rank in the Chow-Witt group as multiples of the hyperbolic element $h.$  A direct application of this construction is the count for the del Pezzo surfaces.
\end{abstract}

\tableofcontents
\newpage
\addtocontents{toc}{\protect\setcounter{tocdepth}{1}}% cache les sous-sections dans la table des matieres
\section{Introduction}
\subsection*{The idea and context}
The aim of this work is to give an enumerative invariant with values in quadratic forms for the count of lines in smooth del Pezzo surfaces of degree 2 and 4. This idea first came from the quadratically enriched count of lines in cubic surfaces, done by Kass and Wickelgren in \cite{Kass_Wickelgren}, and a lot of results are already known in the "quadratically enriched" intersection theory of Chow-Witt rings. 

A difficulty to define and compute these quadratically enriched counts is related to certain orientability properties (c.f. Definition \ref{DEF - orientabilite relative}), a similar problem as for real manifolds. Under suitable orientability hypotheses, works like \cite{BW_excess}, \cite{Kass_Wickelgren}, \cite{Levine_Pauli} or \cite{LevineAB24} gave answers and techniques encompassing both known real and complex invariants (e.g. the cubic surface: the quadratic value $3 + 12h$ gives both the $27$ of the complex count and the 3 as difference between elliptic and hyperbolic lines of \cite{FK_real-cubic}).

When the geometric construction is not orientable, some computations were done by adding more constraints to the original problem, as in \cite{highly_tangent_lines}, \cite{darwin22} or \cite{28-bitangents}. More explicitly, in \cite{28-bitangents}, the authors remove a particular divisor to have relative orientability of the vector bundle on the open complement of the divisor and make a quadratic count with respect to this choice. The count depends on the fact that the zero locus of sections of the vector bundle doesn't meet the divisor. Thus the enumerative problem is slightly changed compared to just counting lines. Similarly, this choice in \cite{darwin22} is reflected by changing the vector bundle to add this open condition.

Instead of making more choices to end up in the Grothendieck-Witt group of quadratic forms and have an Euler degree, here we compute the Euler class in the cohomology group and give as much interpretation as possible directly there. 

\subsection*{The work and results}
A smooth del Pezzo surface $S$ of degree 2 can be seen as a double cover of $\eP^2$ ramified along a smooth quartic curve. If the base field $\fF$ is algebraically closed, the 56 lines appear as inverse images of the 28 bitangents to the quartic (c.f. \cite[Section 8.7]{Dolgachev}). Our strategy to define a possible quadratically enriched count of lines in such surfaces follows the one in \cite{Kass_Wickelgren} where they realise the Hilbert scheme of lines of $S$ as the zero locus of a global section of some locally free sheaf over a suitable variety. We first describe $S$ as a hypersurface in a projective bundle $P$ over $\eP^2$ (c.f. Proposition \ref{PROP - surface deg 2 dans fibre projectif}). Doing so puts us in a similar setup as \cite{Kass_Wickelgren}. Indeed, for a cubic surface in $\eP^3$, one can describe the lines of the surface as a subvariety in the Grassmannian of lines in $\eP^3$ using the incidence correspondence (c.f. Lemma \ref{LEM - schema incidence grassmannienne}):
\[\xymatrix{ &I \simeq \eP(\cQ) \ar[ld]_p \ar[rd]^q &\hspace{-0.1cm} \\ \eP^3&& \hspace{-0.1cm}\Grass(4,2).}\]
with $\cQ$ the canonical quotient sheaf of the Grassmannian. Since the cubic surface is given by the zero locus of a global section of $\cO_{\eP^3}(3)$, its lines are given by the zero locus of a global section of $q_\ast p^\ast \cO_{\eP^3}(3)$ which isomorphic to $\Symalg[3]{\cQ}.$ 

However, the Hilbert scheme of lines in $P = \eP(\cO_{\eP^2}\oplus \cO_{\eP^2}(2))$ is a bit more tricky to define (c.f. Proposition \ref{PROP - schema Hilbert droites de P}) and the incidence scheme is hard to characterise. In the end, this pull-push does not give an explicit description of the locally free coherent sheaf induced. We could only deduce its rank (c.f. Proposition \ref{PROP - dimension du fibre vectoriel voulu}). To summarize, we end up with a diagram:
\[\xymatrix{&& \fU \ar[ld]_{\pi_1} \ar@{}[rd]^(.1){}="a"^(.9){}="b" \ar^{\pi_2} "a";"b" & \hspace{-0.2cm}\\
S \ar@{^{(}->}[r] \ar[rdd] & P \ar[dd]^p &&\hspace{-0.2cm} Y \ar@<-0.1cm>[dd]_q\\
&&I\ar[ld]_{q_1} \ar[rd]^{q_2} &\hspace{-0.2cm}\\
&\eP^2 &&\hspace{-0.2cm} \Grass(3,2)
}\]
with $S$ as the zero locus of a section of $\cO_P(2)$, the projective bundle $Y \simeq \eP(\cO_{\Grass(3,2)}\oplus \Symalg[2]{\cQ^\vee})$ the Hilbert scheme of lines of $P$ and $\fU$ the universal family. The global section defining $S$ gives a section of $\cM = \pi_{2,\ast}\pi_1^\ast \cO_P(2),$ a locally free sheaf of rank $5$ on $Y$ (c.f. Corollary \ref{COR - droites comme zero section fibre}).

Over an algebraically closed field, we can fully describe $\cM$ in a similar way as \cite[Proposition 2.2]{Tih_1980} and it is enough for the Chow group. In this paper, we work over a non algebraically closed field, so we don't have an explicit description of $\cM$ or $\det \cM$. Thus we had to check all possibilities (up to quadratic equivalence) to describe the group $\Chowtilde[5]{Y,\det \cM}$ in which we end up. It only depends on the orientability of $\cM$ and the possibility to take the quadratic degree or not. Let us recall what we mean by (relative) orientability (c.f. Definition \ref{DEF - orientabilite relative}):
\begin{defintro}[Relative orientability]
A locally free sheaf $\cE$ of rank $\dim X$ on $X$ is relatively orientable over $X$ if there exist an invertible sheaf $\cL$ on $X$ such that $\det \cE \simeq \omega_X \otimes \cL^{\otimes 2}.$
\end{defintro}
Satisfying this definition is the only way to have a canonical meaning to the quadratic degree of $e(\cM).$ So the end result is one of the following:
\begin{enumerate}
    \item If $\cM$ is not orientable, then we have a description of the whole Chow-Witt group
\begin{propintro}[Lemma \ref{PROP - chowtilde non orientable est chow} below]
If $\det \cM$ is not orientable on $Y$, there is no quadratic information in $\Chowtilde[5]{Y,\det \cM}.$ Namely, there is a natural isomorphism 
\[\Chowtilde[5]{Y,\det \cM} \isomr \Chowct[5]{Y}.\]
\end{propintro}
The problem is that the inverse image of 1 is hard to describe properly since it needs the choice of an orientation that we don't have in general. The really interesting fact is that a class of even degree in $\Chowct[5]{Y}$ has an explicit description in $\Chowtilde[5]{Y,\det \cM}.$ Indeed, the multiplication by the hyperbolic element $h$ gives the inverse image for classes of even degree (c.f. Proposition \ref{PROP - foncteur H bien def indep torsion}).

    \item If $\cM$ is orientable, then things work out a bit differently. We can compute the Euler degree (i.e. the degree of the Euler class) and using \cite[Proposition 4.4]{Levine_euler-enumerative}, we get that the degree has to be a multiple of $h$ in $\GW(\fF)$. 
\end{enumerate}

We put all cases together and with the fact that the top Chern class of $\cM$ can be determined (c.f. Lemma \ref{LEM - classe chern M sur C}), we finally get 
\begin{thintro}[Main Theorem \ref{THO - classe Euler del Pezzo deg 2}]
A quadratic enrichment of the 56 lines in a del Pezzo surface of degree 2 is:
\begin{enumerate}
    \item if $\cM$ is not orientable
    \[e^{CW}(\cM) = 28h \in \Chowtilde[5]{Y,\det \cM}\]
    \item if $\cM$ is orientable
    \[\deg^{CW} e^{CW}(\cM) = 28h \in \GW(\fF).\]
\end{enumerate}
\end{thintro}

We apply the same technique to get a similar result for del Pezzo surfaces of degree 4 (but this time we know explicitly the locally free coherent sheaf over the Grassmannian):
\begin{thintro}[Proposition \ref{THO - classe Euler del Pezzo degre 4} below]
A quadratic enrichment of the 16 lines in a degree 4 del Pezzo surface is
\[8h \in \Chowtilde[6]{\Grass(5,2),\cO} \simeq \Chowct[6]{\Grass(5,2)}.\]
\end{thintro}

These results, and especially the description of the Chow-Witt groups in maximal codimension when the invertible sheaf is not orientable, indicate that no non-trivial enrichment of the enumerative problem can be found by these methods. The degree 2 and 4 del Pezzo surfaces are in fact special cases of a more general statement for Chow-Witt groups of projective bundles over Grassmannians. It is contained in the following key result:

\begin{thintro}[Theorem \ref{THO - classe Euler non orientable} below]
Let $X$ be the $d-$dimensional Grassmannian $\Grass(k,n)$. Assume either
\begin{itemize}
    \item $n$ is odd and $Y = \eP(\cF)$ is a projective bundle of odd rank $r$ over $X$ or
    \item $Y = \eP(\cF)$ is a projective bundle of even rank $r$ over $X$.
\end{itemize}
Let $\cE$ be a coherent sheaf, locally free of rank $d+r$ over $Y$. Assume $\cE$ is not orientable over the base field $\fF$. 

Then $e^{CW}(\cE)$ is uniquely determined by its value in the Chow group.

Furthermore, if $\deg c_{r+d}(\cE)$ is even, one can write 
\[e^{CW}(\cE) = \frac{\deg c_{r+d}(\cE)}{2}h \in \Chowtilde[r+d]{Y,\det \cE}\simeq \Chowct[r+d]{Y}.\]
\end{thintro}
In this result, we emphasise on the isomorphism between Chow and Chow-Witt groups. It is the reason why we say that there is no enrichment in the non orientable case.

\subsection*{Layout of the paper}

An important part of the paper is section \ref{SEC - motivic setup}, consisting of reminders about the motivic setup. There, we only state the definitions and properties we need to define the Euler class in cohomology, which is the main goal of the section. At the end, we clarify some terminology about orientations and especially the difference between oriented theories and relative orientability of vector bundles.

In section \ref{SEC - calcul classe Euler}, the goal is to define the useful relations and operators between $K-$theories to have maps between the associated cohomologies. We obtain a trivialisation of the Chow-Witt cohomology for a type of non-orientable projective bundle over the Grassmannian and, in particular, a way to define cycles without the data of orientations. It is synthesised in Theorem \ref{THO - classe Euler non orientable}.

The last two parts are the geometrical constructions we need to express the enumerative problem of lines in a smooth del Pezzo of degree 2 and 4 as the zero locus of sections of locally free sheaves over some smooth and proper scheme.

\subsection*{Notations and conventions}

In the whole paper $\fF$ will refer to a perfect field of characteristic different from $2.$ Straight letters will be used to describe schemes and vector bundles, curvy letters will be reserved for sheaves. If $X$ is a scheme, $\Omega_X$ stands for its sheaf of Kähler differentials and $\omega_X = \det \Omega_X.$

To define a vector bundle from a locally free coherent sheaf $\cE$ of finite rank, we will stick to EGA's convention $E = \eV(\cE) = \Spec{\Symalg{\cE}}$ and $\eP(\cE) = \Proj{\Symalg{\cE}}.$ In particular, there will sometime be a dualization compared to cited results using the other convention.

\subsection*{Acknowledgments}
The author would like to thank his PhD advisors Adrien Dubouloz and Jan Nagel for their full support, help and the opportunity to accomplish this work. The author is also very grateful to Pedro Montero for pointing out the geometric construction for the degree two del Pezzo surfaces.

\addtocontents{toc}{\protect\setcounter{tocdepth}{2}}%a partir d'ici, la table des matieres montre jusqu'aux sous-sections
\section{Motivic setup}\label{SEC - motivic setup}
This first "reminder" part is mostly taken from \cite{DJK21}. We slightly adapt the notations.

\subsection{Stable homotopy and six functors}
In the following, the base schemes will be quasi-compact and quasi-separated and we call \emph{s-morphism}\index{s-morphisme} a separated morphism of finite type.

Let $S$ be a quasi-compact and quasi-separated noetherian scheme of finite dimension. Then we write $\StableHo{S}$ the stable $\infty-$category of motivic spectra.

The unit for $\otimes$ in $\StableHo{S}$ is noted $\eS_S$ and is the stabilisation of $\eUn_S$.

\begin{prop}[\cite{Ayoub_6foncteurs}]
For every morphism $f: T \to S$, we get a pair of adjoint functors:
\[\xymatrix{f^\ast: \StableHo{S} \ar@<+2pt>[r] & \StableHo{T}: f_\ast, \ar@<+2pt>[l]}\]
called inverse image (by left Kan extension) and direct image. Moreover, if $f$ is a s-morphism, then we get another adjoint pair:
\[\xymatrix{f_!: \StableHo{T} \ar@<+2pt>[r] & \StableHo{S}: f^!, \ar@<+2pt>[l]}\]
called direct (and inverse) exceptional image.
\end{prop}

\begin{tho}[Six functors on $\StableHo{S}$, \cite{Ayoub_6foncteurs}]\label{THO - formalisme six foncteurs SH}\index{Six foncteurs sur $\StableHo{S}$}
The six operations $(\otimes,\underline{\operatorname{Hom}},f^\ast,f_\ast,f_!,f^!)$ satisfy to the six functors formalism:
\begin{enumerate}
\item For every morphism $f$, $f^\ast$ are symmetric monoidal.
\item There is a natural transformation $f_! \to f_\ast$ invertible when $f$ is proper.
\item There is an invertible natural transformation $f^\ast \to f^!$ when $f$ is an open immersion.
\item There is a canonical isomorphism
\[\eE \otimes f_!(\eF) \to f_!(f^\ast(\eE)\otimes \eF)\]
for all s-morphisms $f: T \to S$ and $\eE$ in $\StableHo{S}$, $\eF$ in $\StableHo{T}$.
\item For every Cartesian square
\[\xymatrix{T' \ar[r]^g \ar[d]_q & S' \ar[d]^p \\ T\ar[r]_f & S,}\]
if $f$ and $g$ are s-morphisms, then we get canonical isomorphisms:
\[p^\ast f_! \to g_!q^\ast, \qquad q_\ast g^! \to f^!p_\ast.\]
\end{enumerate}
\end{tho}

\begin{deff}[Suspension by a locally free sheaf]\label{DEF - suspension par faisceau localement libre}
To a locally free sheaf $\cE$ of finite rank over $S$, we can associate an auto-equivalence 
\[\Sigma^\cE : \StableHo{S} \to \StableHo{S}\]
called the suspension by $\cE.$
\end{deff}

This leads to the generalisation
\begin{prop}[Thom space of a locally free sheaf]\label{PROP - def foncteur Thom}
Let us define $\operatorname{Th}$ the functor which, to $\cE$ a locally free of finite rank sheaf over $S$, associates $\Thom[S]{\cE} := \Sigma^\cE(\eS_S)\in \Picard{\StableHo{S}}.$ It can be naturally extended as a morphism of $\cE_\infty$-groups:
\[\operatorname{Th}: K(-) \to \Picard{\StableHo{-}}.\]
\end{prop}

\begin{rmq}
Rather than considering vector bundles, when possible we chose to keep the suspension by Thom spaces of locally free sheaves. Thus, we replace the $K-$theory of vector bundles by the one of locally free sheaves (Grothendieck $K_0$ group).
\end{rmq}

\begin{rmq}
Let $\cE'\to \cE \to \cE''$ be an exact triangle of perfect complexes (short exact sequence of vector bundles). It canonically gives rise to an isomorphism $\qscal{\cE} \simeq \qscal{\cE'} + \qscal{\cE''}$ in $K_0(S)$ with $\qscal{\cE}$ the class of $\cE$ in $K_0(S)$ (we will most of the time just write $\cE$). We then get $\Thom[S]{\cE} \simeq \Thom[S]{\cE'}\otimes\Thom[S]{\cE''}$ in $\StableHo{S}$.
\end{rmq}

\begin{tho}[Homotopic purity]\label{THO - purete version six foncteurs}
If $f$ is a smooth s-morphism, there exists a canonical isomorphism of functors 
\[\fp_f: \Sigma^{\cT_f}f^\ast \to f^!\]
with $\cT_f$ the relative tangent sheaf of $f$.
\end{tho}

\begin{cor}
If $f$ is an étale s-morphism, then $\fp_f$ becomes an isomorphism $f^\ast \simeq f^!$ (from the natural transformation already present with the six functors \ref{THO - formalisme six foncteurs SH}).
\end{cor}

\begin{cor}
From $\fp_f,$ one recovers the left adjoint of $f^\ast$. It is then of the form $f_\sharp \simeq f_!\Sigma^{\cT_f}$. 
\end{cor}

\subsection{Universal Euler class}

We here do the same construction as \cite[Paragraph 3.1]{DJK21} for the universal Euler class. However, we directly build it as an element of the stable homotopy.\\

Let $X \to S$ be a smooth s-morphism and $\cE$ locally free sheaf of constant rank $r$ on $X$. We then get schemes morphisms $p: E = \eV(\cE) \to X$ and $s: X \to E$, the projection from the vector bundle and the zero section, respectively. The relative tangent sheaf to the zero section $\cT_s$ is 
the conormal sheaf $(\cI_{s(X)}/\cI_{s(X)}^{2})$.
Moreover, since $\Symalg{\cE} \simeq \Symalg{\cT_s}$ it gives rise to the purity isomorphism \ref{THO - purete version six foncteurs}: $\Sigma^{\cE} s^\ast \isomr s^!$.

\begin{deff}[Stable universal Euler class]\label{DEF - classe Euler univ stable}
Given the above notations, the counit $\epsilon$ of the adjonction $\xymatrix{s^\ast: \StableHo{E} \ar@<+2pt>[r] & \StableHo{X}: s_{\ast}, \ar@<+2pt>[l]}$, the unit $ \eta$ of the adjonction $\xymatrix{s_!: \StableHo{X} \ar@<+2pt>[r] & \StableHo{E}: s^{!}, \ar@<+2pt>[l]}$ and the purity Theorem \ref{THO - purete version six foncteurs} together give the stable universal Euler class as the composition:
\[\xymatrix{e^u(\cE): \eS_X \ar[r]^\eta & s^! s_{!}\eS_X \ar[r]^{\fp_{s}\qquad \quad } & \Thom[X]{\cE}\otimes_X s^\ast s_{\ast}\eS_X \ar[r]^{\qquad\epsilon}&\Thom[X]{\cE} .}\]
\end{deff}

\begin{rmq}
For now, it is not a class in a cohomology nor is it question of the orientation of a vector bundle. This morphism exists the moment that we have a vector bundle $E= \eV(\cE) \to X$.
\end{rmq}

\subsection{Homologies, cohomologies}

In the context of the stable homotopy and six functors formalism, we define homologies / cohomologies as morphism classes:

\begin{deff}[Homologies and cohomologies]\label{DEF - homologies et cohomologies}
Let us fix a base scheme $S$ and $\eE$ a motivic spectrum in $\StableHo{S}$. For $f: X \to S$ a morphism and $v$ in $K(X)$, we define:
\begin{align*}
\text{(Borel-Moore homology)}\ &\ \eE^{BM}(X/S,v):= \Maps[\StableHo{S}]{\eS_S,f_\ast(f^!(\eE) \otimes\Thom[X]{v})}\\
\text{(cohomology)}\ &\ \eE(X,v):= \Maps[\StableHo{S}]{\eS_S,f_\ast(f^\ast(\eE) \otimes\Thom[X]{v})}\\
\text{(homology)}\ &\ \eE^c(X/S,v):= \Maps[\StableHo{S}]{\eS_S,f_!(f^!(\eE) \otimes\Thom[X]{v})}\\
\text{(compactly supported cohomology)}\ &\ \eE_c(X,v):= \Maps[\StableHo{S}]{\eS_S,f_!(f^\ast(\eE) \otimes\Thom[X]{v})}.\\
\end{align*}
\end{deff}

\begin{rmq}
For $n$ an integer, we will write $ \eE^{BM}_n(X/S,v):= \Maps[\StableHo{S}]{\eS_S[n],f_\ast(f^!(\eE) \otimes\Thom[X]{v})}$ and similarly $\eE^n(X,v):= \Maps[\StableHo{S}]{\eS_S,f_\ast(f^\ast(\eE) \otimes\Thom[X]{v})[n]}.$
\end{rmq}

We have a preferred spectrum in  $\StableHo{S},$ it is the unit $\eS_S$ and its cohomology will be written as $H(X,v).$ When $\eE$ is a spectrum, the cohomology it represents will have the same properties than the universal cohomology under the choice of a unit morphism $\rho: \eS_S \to \eE$. Then, $\eE$ will be called a unital theory (c.f. \cite[Definition 4.1.2]{DJK21}).

\begin{deff}[Euler class with coefficients]\label{DEF - Euler class with coefficients}
Let $f: X \to S$ be a smooth s-morphism and $\cE$ a locally free sheaf of rank $r$ on $X$. Then, for $\eE$ a unital theory on $S$, the Euler class \ref{DEF - classe Euler univ stable} $e^u(\cE)$ defines an element $e^\eE(\cE)$ of $\eE(X,\cE)\simeq \eE^{BM}(X/S,\cE-\cT_f).$ We call this element the Euler class with coefficients (in $\eE$) and write $e(\cE)$ if there is no ambiguity on $\eE.$
\end{deff}

\begin{prop}[Proper push-forward]\label{PROP - covariance propre}
Let $\eE$ be a spectrum in $\StableHo{S}.$ For every proper morphism $f: X \to Y$ of s-schemes on $S,$ there is a direct image 
\[f_\ast: \eE^{BM}(X/S,f^\ast v) \to \eE^{BM}(Y/S,v).\]
\end{prop}

\begin{proof}
If we note $p,q$ the structure morphisms of $X$ and $Y$ on $S,$ we get the commutative diagram
\[\xymatrix{X\ar[rr]^f \ar[dr]_p && Y \ar[dl]^q \\ &S.&}\]
Then,
\begin{align*}
    f_\ast \Sigma^{f^\ast v}p^! \eE & \simeq \Sigma^v f_\ast p^! \eE\\
    & \simeq \Sigma^v f_!f^!q^!\eE \to \Sigma^v q^!\eE
\end{align*}
where the last morphism is the counit of the adjonction $s_! \dashv s^!.$
\end{proof}

\subsection{Orientations}\label{SSEC - orientations}

We now define two types of orientations that will be useful in the following applications: the orientation of a (co)homological theory and the (relative) orientation of a vector bundle over its base.

\begin{deff}[Oriented theory]\label{DEF - theorie orientee}
Let $S$ be a base scheme. A theory $\eE$ in $\StableHo{S}$ is said 
\begin{itemize}
    \item ($\Gl-$)oriented if $\eE^{\ast}(X,\fvirt{\cE}) \simeq \eE^{\ast+r}(X)$
    \item $\SL-$oriented if $\eE^{\ast}(X,\fvirt{\cE}) \simeq \eE^{\ast+r}(X,\fvirt{\det \cE})$
    \item $\SL^c-$oriented if $\eE^{\ast}(X,\fvirt{\cE}) \simeq \eE^{\ast+r}(X,\fvirt{\det \cE \otimes \cL^{\otimes 2}})$
\end{itemize}
for all smooth $S-$scheme $X$, $\cE$ locally free of rank $r$ sheaf on $X$ and $\cL$ an invertible $\cO_X-$module.
\end{deff}

\begin{rmq}
The isomorphisms in the above definition must be functorial. When it is the case, we will then say that the theory is orientable and a choice of isomorphisms fixes the orientation. In some cases, like the Milnor-Witt theory, there is existence and unicity of the orientation (c.f. \cite{Ananyevskiy_SLoriented}) hence no distinction in the definition between orientable and oriented.
\end{rmq}

\begin{prop}[Orientation of usual theories]\label{PROP - orientation theories usuelles}
Let $S = \Spec{\fF}$ be the base scheme.
\begin{itemize}
    \item The Milnor-Witt cohomology spectrum is $\SL^c-$oriented in $\StableHo{\fF}$ (c.f. \cite[Remark 5.5]{Ananyevskiy_SLoriented}).
    \item The Witt cohomology spectrum (or of the $\fId^\bullet-$cohomology) is $\SL^c-$oriented in $\StableHo{\fF}$.
    \item The Milnor cohomology spectrum is $\Gl-$oriented in $\StableHo{\fF}$.
\end{itemize}
\end{prop}

In the case of an oriented theory, we can look at what becomes the Euler class of a locally free sheaf $\cE$ of rank $r$ on $X$. In the particular case of the Milnor theory, the Euler class becomes the top Chern class of the sheaf. If $\dim X = r,$ this $c_r(\cE)$ is a $0-$dimensional class and we can take its degree which is an integer. In the Milnor-Witt case, we would like to similarly take the degree of the $0-$dimensional class induced by the Euler class and look at the value in the Grothendieck-Witt group. 

More explicitly, we have $e(\cE) \in \Chowtilde[r]{X,\det \cE}$ and  $\Chowtilde[r]{X,\det \cE}$ is isomorphic to $\Chowtildecov[0]{X,\det \cE \otimes \omega_X^\vee}.$ So we need to first define

\begin{deff}[Relative orientability]\label{DEF - orientabilite relative}
Let $X$ be a smooth $\fF-$scheme and $\cE$ a locally free sheaf of finite rank $r = \dim X$ on it. $\cE$ is relatively orientable if there exists $\cL$ a locally free sheaf of rank 1 such that $\det \cE \simeq \omega_X \otimes \cL^{\otimes 2}.$ 

We also say that two invertible sheaves are in the same quadratic equivalence class or quadratically equivalent if they differ only by the product with a tensor square and we will note $\cE \simquad \cF$.
\end{deff}

\begin{deff}[Quadratic degree map]\label{DEF - degre quadratique}
Let $\pi: X\to \Spec{\fF}$ be a smooth and proper $\fF-$scheme. The quadratic degree map is the push forward (c.f. \ref{PROP - covariance propre})
\[\pi_\ast: \Chowtildecov[0]{X} \to \Chowtildecov[0]{\fF}\simeq \GW(\fF).\]
\end{deff}

\begin{rmq}\label{RMQ - degre quad classe Euler}
To talk about the quadratic degree of an Euler class, it is necessary for the sheaf to be relatively orientable over the base scheme.

We will also talk about relative orientation of $\Chowtilde[\dim X]{X,\cL}$ if $\cL \simquad \omega_X.$ 
\end{rmq}

Since most of the computations done here are for projective bundles and Grassmannians, we will discuss some vector bundle classes as examples.

\begin{prop}[Orientability over the Grassmannian]\label{PROP - orientabilite grassmann}
Let $G = \Grass(n,k)$ be the Grassmannian of the rank $k$ quotients of $V$ a $n-$dimensional vector space over $\fF$. Then, $\omega_G \simeq (\det \cQ^\vee)^{\otimes n} := \cO(-n)$ in $\Picard{G}$. Furthermore, if with denote the universal quotient by $Q$, then $\det\left(\Symalg[r]{\cQ}\right)$ is isomorphic to $\cO\left(rk (\Symalg[r]{\cQ})\right) = \cO\left(\binom{r+k-1}{k-1}\right).$
\end{prop}

\begin{proof}
Using the description of the tangent space of the Grassmannian given in \cite[Example 16.1]{Harris_alg-geom} and taking its dual to have the Kähler differential sheaf, we know that $\Omega_G \simeq \Hom{\cQ,\cS} \simeq \cQ^\vee\otimes \cS$ with $\cS$ the associated sub-sheaf to $\cQ$ in the Euler exact sequence 
\[\xymatrix{0 \ar[r] & \cS \ar[r] & V\otimes\cO_G \ar[r] & \cQ \ar[r] & 0.}\]
Moreover, $\det \cS \simeq \det \cQ^\vee$ and with the Plücker embedding $\det \cQ :=\cO(1)$ as a generator of the Picard group. Hence, since $rk(\cQ) = k$ and $rk(\cS) = n-k,$ we get $\omega_G \simeq (\det \cQ^\vee)^{n-k}\otimes (\det \cS)^{k} \simeq \cO(-n).$ 
The second part comes from $\det\left(\Symalg[r]{\cQ}\right)\simeq (\det \cQ)^{\binom{r+k-1}{k-1}}.$
\end{proof}
\begin{rmq}
From the last proof, we can emphasise on the fact that we take $\det \cQ$ as a generator of the Picard group of the Grassmannian and we note it $\cO_G(1).$
\end{rmq}

Let $\cF$ be a locally free $\cO_X-$module of finite rank over a scheme $X$. Then $ \pi: Y = \eP(\cF)\to X$ then satisfies the two exact sequences:
\[0 \to \Omega_{Y/X} \to \cO_Y(-1)\otimes\pi^\ast\cF \to \cO_Y \to 0 \]
\[\text{and}\ 0 \to \pi^\ast \Omega_X \to \Omega_Y \to \Omega_{Y/X} \to 0.\]
Thus
\begin{align}\label{EQU - omega fibre projectif}
\omega_Y \simeq \cO_Y(-rg\ \cF)\otimes \pi^\ast \det \cF \otimes \pi^\ast \omega_X.
\end{align}

\section{Euler class computation}\label{SEC - calcul classe Euler}

We will apply all this motivic setup in the stable Witt, $(2-)$Milnor or Milnor-Witt cohomology spectra. Let us begin with a short reminder of definitions since we will need to use them a bit (properties and characterisations can be found in \cite{deglise_KMW_notes}). Just before that, we will remind \cite[Definition 3.1.2, Remark 3.1.4]{DF_KMW-rpz}:

\begin{prop}
The (co)homologies with coefficients in Milnor-Witt (or its quotients) $K-$theory are representable as Eilenberg-MacLane spectra in $\StableHo{\fF}.$
\end{prop}

\begin{rmq}
In particular, the Euler class exists in all cohomologies by specialisation and coincides with cohomology change.
\end{rmq}

\subsection{Definitions of the $K-$theory rings}

\begin{deff}[Milnor $K-$theory ring]\label{DEF - K-th de Milnor}
For $\fF$ a field, the Milnor $K$-theory ring $K^M_\ast (\fF)$ is defined by the $\eZ-$graded algebra generated by the symbols $\accol{a}$, for $a$ in $\fF^\times$, in degree $+1$ and relations:
\begin{enumerate}[label={(M\arabic*)}]
\item $\accol{a,1-a} = 0$
\item $\accol{ab} = \accol{a} + \accol{b}$
\end{enumerate}
with the convention $\accol{a_1, \ldots,a_n} = \accol{a_1} \ldots \accol{a_n}$.
\end{deff}

\begin{rmq}
We have two obvious identities: $K^M_0(\fF) = \eZ$ and $K^M_1(\fF) = \fF^\times$.
\end{rmq}

\begin{deff}[Milnor-Witt $K-$theory ring]\label{DEF - K-th de Milnor-Witt}
Let $\fF$ be a field. The Milnor-Witt $K$-theory ring is given by the $\eZ-$graded algebra generated by elements $[a]$ in degree $+1$ for $a$ in $\fF^\times$ and $\eta$ in degree $-1$, called Hopf element, with the relations:
\begin{enumerate}[label={(MW\arabic*)}]
\item \label{ITM - axm MW1} $[a,1-a] = 0$ (if $a$ and $a-1$ are not $0$),
\item \label{ITM - axm MW2}$[ab] = [a] + [b] + \eta[a,b]$,
\item \label{ITM - axm MW3}$\eta [a] = [a]\eta$,
\item \label{ITM - axm MW4}$\eta (2 + \eta[-1]) = 0$
\end{enumerate}
with the convention $[a_1,\ldots,a_n] = [a_1]\ldots [a_n]$ in $\KMW[n]{\fF}.$
\end{deff}

\begin{deff}[Grothendieck-Witt ring]\label{DEF - anneau de Grothendieck-Witt}
We note $\GW(\fF)$ the ring of quadratic forms up to isometry with the multiplication being the tensor product of quadratic forms and the underlying group being the Grothendieck group spanned from the direct sum of quadratic forms.
\end{deff}

For a scalar $a$ in $\fF^\times,$ one can look at the quadratic form $\qscal{a}: X \mapsto a X^2.$ 

\begin{rmq}
We can define $\GW(\fF)$ by generators in $\fF^\times/(\fF^\times)^2$ and relations $\qscal{u}+ \qscal{v} = \qscal{u+v,(u+v)uv}$ as long as $u+v \neq 0$.
\end{rmq}

\begin{lemme}[\cite{Morel_A1-alg-top}, Lemma 3.10]\label{LEM - iso GW et KMW0}
The ring homomorphism 
\[\begin{array}{ccc}
    \GW(\fF) & \to & \KMW[0]{\fF}  \\
    \qscal{a} & \mapsto & 1 + \eta [a] 
\end{array}\] 
is an isomorphism. The image of $\qscal{a}$ in $\KMW[0]{\fF}$ will still be noted $\qscal{a}$.
\end{lemme}

\begin{rmq}
The element $h = \qscal{1}+ \qscal{-1}$ is called the hyperbolic form in the Grothendieck-Witt group. It becomes $h = \qscal{-1} + 1$ in $\KMW[0]{\fF}$ and verifies $\eta h = 0,\ h^2 = 2h$. It is an important element to link the quadratic degree and the classical one as we will see.
\end{rmq}

\begin{deff}[Witt ring]\label{DEF - anneau de Witt}
We note $\Witt(\fF) = \GW(\fF)/(h)$ the quotient ring of quadratic forms by the ideal generated by $h.$
\end{deff}

\begin{deff}[Fundamental ideal in $\GW(\fF)$]\label{DEF - ideal fonda de GW(k)}
The kernel of the ring homomorphism $rk: \GW(\fF) \to \eZ$ is the fundamental ideal of $\GW(\fF)$ and is noted $\fId(\fF)$. 

By convention, for any $n$ non-positive integer $\fId^n(\fF) = \Witt(\fF).$
\end{deff}

\begin{deff}[Witt $K-$theory ring]\label{DEF - K-th de Witt}
For $\fF$ a field, we define the Witt $K-$theory ring by:
\[K^W_\ast(\fF) = \KMW{\fF}/(h).\]
\end{deff}

\begin{tho}\label{THO - isomorphisme Kth Witt et ideal fondamental}
Let $\fF$ be a field of characteristic different from $2$. There exists a unique isomorphism $\psi$ of $\eZ-$graded algebras
\[\xymatrix{K^W_\ast(\fF) \ar[rr]^\psi  && \fId^\ast(\fF)}\]
sending the quotient class of $\eta$ to $\eta$ and the quotient class $\accol{a}$ of $-[a]$ to $\qqscal{a}$ with $\qqscal{a} := 1-\qscal{a}$.
(c.f. \cite[Théorème 2.4, Remarque 5.2]{Morel_id-fonda})
\end{tho}

We can also extend the definition to obtain the twisted $K-$theory: 

\begin{deff}[Twisted Milnor-Witt $K-$theory ring]\label{DEF - kth MW tordue}
Let $\cL$ be a 1-dimensional $\fF-$vector space and $n$ an integer, then 
\[\KMW[n]{\fF,\cL} = \KMW[n]{\fF}\otimes_{\eZ[\fF^\times]} \eZ[\cL^\times]\] 
where $\KMW{\fF}$ is a $\eZ[\fF^\times]-$module for the multiplication $a\cdot m = \qscal{a}m$ for $a \in \fF^\times$ and $m \in \KMW{\fF}$ and $\eZ[\cL^\times]$ is a module by scalar multiplication from $\fF^\times$ to $\cL^\times$.
\end{deff}

\begin{rmq}
We also extend the twisted version to the Witt $K-$theory and powers of the fundamental ideal.
\end{rmq}

An other important consideration is the Witt-valued cohomology, which is different from the cohomology with coefficients in $\fId^\ast$. The coefficients are denoted by $\cW(\cL)$ or $\bfW(\cL)$ in the literature and we will use the following characterisation:

\begin{deff}[Witt twisted coefficients]\label{DEF - coeff Witt twistes}
Let $\cL$ be a 1-dimensional $\fF-$vector space and $n$ an integer, then 
\[\bfW_n(\fF,\cL) = K^{MW}_n[\eta^{-1}](\fF)\otimes_{\eZ[\fF^\times]} \eZ[\cL^\times]\] 
where the $\eZ[\fF^\times]-$module structures are induced from the ones of \ref{DEF - kth MW tordue}.
\end{deff}
One has to be careful between $\bfW_n(\fF,\cL)$ and $K^W_n(\fF,\cL)$ because they will in general differ a bit. The aim of the next part is precisely to clarify the relations.

\subsection{Relations between $K-$theory rings}\label{SSEC - digression sur H}

To compute an Euler class in the Chow-Witt group and its degree in $\GW(\fF)$, we will break it down to a computation in the Chow group and $\fId^\bullet-$cohomology.

\begin{prop}\label{PROP - isomorphisme K-Milnor et K-MW/eta}
There is an isomorphism of $\eZ-$graded algebras between Milnor and Milnor-Witt $K$-theory rings when quotienting the later by the ideal spanned by $\eta$:
\[E : \left|\begin{array}{ccc}
    \KMW{\fF}/(\eta)& \isomr & K^M_\ast(\fF)  \\ 
    \left[a\right] & \mapsto & \accol{a}.
\end{array}\right.\]
\end{prop}

\begin{proof}
The generators are the same and the Milnor-Witt relations from Definition \ref{DEF - K-th de Milnor-Witt} simplify into the ones of the Milnor $K-$theory ring when quotienting by $\eta.$
\end{proof}

If we want to extend this application to the $K-$theory rings with twists, we need twists on the Milnor $K-$theory ring. Following the previous isomorphism, the $\eZ[\fF^\times]-$module structure on $K^M_n(\fF)$ becomes $a\cdot m = m$ when $\eta$ vanishes (c.f. Definition \ref{DEF - kth MW tordue}).

\begin{cor}\label{COR - kth Milnor ne se tord pas}
The isomorphism $E$ from Proposition \ref{PROP - isomorphisme K-Milnor et K-MW/eta} can be extended to the $K-$theory rings with twists into 
\[E_\cL : \left| \begin{array}{ccc}
    \KMW{\fF,\cL}/(\eta) &\isomr  & K^M_\ast(\fF)  \\
    \left[a\right] \otimes l & \mapsto & \accol{a}.
\end{array}\right.\]
\end{cor}

\begin{proof}
The only thing to check is that $K^M_\ast(\fF,\cL) \simeq K^M_\ast(\fF).$ If $l = a l'$ in $\cL,$ then 
\begin{align*}
    m\otimes l &= m \otimes al'\\
    & = (am)\otimes l'\\
    &= m\otimes l'
\end{align*}
and all the elements are the same.
\end{proof}

In particular, we have the short exact sequence of abelian groups:
\begin{align}\label{EQU - def F: MW vers M K-th}
\KMW[q+1]{\fF,\cL} \overset{\gamma_n}{\to} \KMW[q]{\fF,\cL} \overset{F}{\to} K^M_q(\fF) \to 0
\end{align}
where $\gamma_\eta$ is the multiplication by $\eta$ and $F$ is the quotient map by $(\eta)$ followed by the isomorphism $E_\cL.$

This application $F$ will give rise to a functor whose adjoint is the following construction.

First, let's consider the $\eN-$graded morphism of algebras:
\[\theta: \left|\begin{array}{ccc}
(\fF^\times)^{\otimes \ast} & \to & \KMW{\fF} \\
u_1\otimes\ldots\otimes u_q &\mapsto & h[u_1,\ldots,u_q].
\end{array}\right.
\]

\begin{lemme}
The above application $\theta$ gives rise to a morphism of $K-$theory rings
\begin{align}\label{EQU - def H: M vers MW K-th}
H: K^M_\ast(\fF) \to \KMW{\fF}.
\end{align}
\end{lemme}

\begin{proof}
By looking in degree at least $2$ and computing the image of $(u_1u_2)\otimes v$, we get 
\begin{align*}
    \theta((u_1u_2)\otimes v)& = h[u_1u_2,v]\\
    &= h[u_1,v] + h[u_2,v] + h\eta [u_1,u_2,v]\\
    &= \theta(u_1 \otimes v) + \theta(u_2\otimes v).
\end{align*}
Hence, a factorisation through the relations of Milnor $K-$theory ring and we note $H$ the induced morphism.
\end{proof}

With the relations in the Milnor-Witt $K-$theory one can deduce that $\qscal{u}h = h$ for any $u \in \fF^\times$ and thus an invariance of the multiplication by $h$ at the level of the twisted cohomology:

\begin{prop}\label{PROP - foncteur H bien def indep torsion}
For $\cL$ a 1-dimensional $\fF-$vector space, the twisted hyperbolic application
\[H:\left| \begin{array}{ccc}
     K^M_\ast(\fF)& \to& \KMW{\fF,\cL} \\
    \sigma & \mapsto & (h\sigma)\otimes l 
\end{array}\right. \]
is well defined and does not depend on the choice of $l$ in $\cL^\times$.
\end{prop}

\begin{proof}
It is essentially by definition (c.f. \ref{DEF - kth MW tordue}). Choosing another vector is choosing a scalar $a$ such that $l' = al$. Hence 
\[(h\sigma)\otimes l' = (h\sigma)\otimes al = (\qscal{a}h\sigma)\otimes l =  (h\sigma)\otimes l.\]
\end{proof}

\begin{deff}[Forgetful and hyperbolic applications]\label{DEF - applications oubli et hyperbolique}
The two $\eZ-$graded algebras homomorphisms $F: \KMW{\fF,\cL} \to K^M_\ast(\fF)$ and $H: K^M_\ast(\fF) \to \KMW{\fF,\cL}$ (c.f. (\ref{EQU - def F: MW vers M K-th}) and Proposition \ref{PROP - foncteur H bien def indep torsion}) are called the forgetful and hyperbolic homomorphisms. They are functorial with respect to the field.
\end{deff}

\begin{rmq}
By definition, $F$ is characterised by $F(\eta) = 0$ and $F([a]\otimes l) = \accol{a}$. $H$ is given by the multiplication by $h$ thus $H(\accol{a}) = h[a]$.

Moreover, the compositions give $F\circ H = 2 \Id$ and $H \circ F = \gamma_h$ (the multiplication by $h$). 
\end{rmq}

\subsection{Cohomology and Rost-Schmid complexes}

This is mostly a very quick reminder of the objects we consider. More details and especially the construction of the residue morphisms can be found for example in \cite[Section 2]{fasel_Chow-witt}.

\begin{deff}[Rost complex]\label{DEF - complexe de Rost}
Let $X$ be a smooth scheme of finite type over $\fF$ a perfect field. For $j$ in $\eZ$, one can define the complex:
\[\xymatrix{R_j(X): \ldots \ar[r] & \underset{p \in X^{(i)}}{\bigoplus} K^M_{j-i}(\kappa(p)) \ar[r]^{d^i_j} & \underset{q \in X^{(i+1)}}{\bigoplus} K^M_{j-i-1}(\kappa(q)) \ar[r] & \ldots}\]
where $X^{(i)}$ is the set of codimension $i$ points in $X$ and $d^i_j$ are residue morphisms.
\end{deff}

\begin{prop}[\cite{Rost_Chowgrps}] 
The cohomology of the Rost complex $A^i(X,j) = \ker(d^i_j)/\img(d^{i-1}_j)$ gives back the Chow groups of codimension $i$ sub-varieties in $X$:
\[\Chowct[i]{X} \simeq A^i(X,i).\]
\end{prop}

For the other cohomologies, we need some twisting by sheaves on points. So we will consider, for $X$ a smooth scheme of finite type over $\fF$ and $x$ a point in $X$:
\begin{itemize}
    \item the normal sheaf $\cN_{x,X} := \left(\cI_x/\cI_x^2\right)^\vee$ with $\cI_x \subseteq \cO_{X,x}$ the ideal sheaf of $x$,
    \item the invertible sheaf $\nu_x$ as the determinant of $\cN_{x,X}$ and
    \item  $\rest{\cL}{x}$ the tensor product $\cL_x\otimes_{\cO_{X,x}} \kappa(x)$ for $\cL$ an invertible $\cO_X-$module.
\end{itemize}

\begin{deff}[Rost-Schmid complex]\label{DEF - cplx de Rost-Schmid}
For $j$ in $\eZ$ and $\cL$ an invertible $\cO_X-$module, one can define the Rost-Schmid complex associated to $X,\ j$ and $\cL$ by:
\[\xymatrix{\ldots \ar[r]&  \RostS[i]{X}{j}{\cL} \ar[r]^{d^i_{X,j,\cL}} & \RostS[i+1]{X}{j}{\cL} \ar[r] & \ldots,}\]
with:
\[\RostS[i]{X}{j}{\cL} = \bigoplus_{x \in X^{(i)}} \KMW[j-i]{\kappa(x), \nu_x\otimes_{\kappa(x)}\rest{\cL}{x}}\]
and $d^i_{X,j,\cL}$ the differential of the complex.
\end{deff}

\begin{rmq}
We will note $\RostSO{X}{j}$ for $\RostS{X}{j}{\cO_X}$. We can also extend the definition of the Rost-Schmid complex to the quotient and localised $K-$theories introduced before.
\end{rmq}

\begin{rmq}\label{RMQ - complexe Rost-Schmid est Rost pour Milnor}
One can see that the Rost-Schmid complex for the Milnor-Witt $K-$theory ring with twists leads to a Rost-Schmid complex on Milnor $K-$theory ring by quotienting by $(\eta).$ From Corollary \ref{COR - kth Milnor ne se tord pas}, there is no need consider twists for Milnor $K-$theory rings and it is just the previous Rost complex.
\end{rmq}

\begin{deff}[Rost-Schmid cohomology groups]\label{DEF - grps de Rost-Schmid}
For $i,j$ integers and $\cL$ an invertible $\cO_X-$module, the $i^{\text{th}}$ cohomology group associated to the Rost-Schmid complex $\RostS{X}{j}{\cL}$ is 
\[\RSgrp{i}{X}{j}{\cL}:= \ker(d^i_{X,j,\cL})/\img(d^{i-1}_{X,j,\cL}).\]
We will also write $\RSgrpO{i}{X}{j}$ for the cohomology of $\RostSO{X}{j}$.
\end{deff}

\begin{deff}[Chow-Witt groups]\label{DEF - groupes de Chow-Witt}
Let $i$ be an integer and $\cL$ an invertible $\cO_X-$module. The  $i^{\text{th}}$ Chow-Witt group associated to $X$ and $\cL$ is:
\[\Chowtilde[i]{X,\cL}:= \RSgrp{i}{X}{i}{\cL}.\]
We will note $\Chowtilde[i]{X}$ for $\Chowtilde[i]{X,\cO_X}$.
\end{deff}

\begin{rmq}
For $i$ negative or greater than $\dim X$, the Rost-Schmid complex is constant. Thus, the cohomology groups $\RSgrp{i}{X}{j}{\cL}$ and $\Chowtilde[i]{X,\cL}$ are trivial.
\end{rmq}

Following this definition, we can change the $K-$theory to the Witt one or the $\fId^\bullet-$cohomology. The fact that the Rost-Schmid complex is still a complex with this change of coefficients can be found for example in \cite[Théorème 9.2.4]{Fasel_Chow-Witt_SMF}. We get 
 \[H^i(X,\fId^i(\cL)) \simeq H^i(X,\underline{K}^W_i(\cL)).\]
 
\begin{rmq}
We will not write $\bfW^i(X,\cL)$ for $H^i(X,\underline{K}^W_i(\cL))$ to avoid the conflict of notations with the cohomology with coefficients in the Witt sheaf, which is isomorphic to $H^i(X,\underline{K}^{MW}_i[\eta^{-1}](\cL))$ (c.f. Definition \ref{DEF - coeff Witt twistes}).
\end{rmq}

\begin{deff}[Witt cohomology groups]\label{DEF - groupes cohomologie Witt}
The Witt cohomology groups of a scheme $X$, with respect to an invertible sheaf $\cL$ on it are the groups
\[\bfW^i(X,\cL):= H^i(X,\bfW_i(\cL))\]
with $\bfW_i(\cL)$ the sheaf associated to $\bfW_i(\kappa(x),\nu_x\otimes_{\kappa(x)}\rest{\cL}{x})$ as for the $K^{MW}$ sheaf. An explanation for how the Gersten-type complex is still a complex when changing coefficients can be found in \cite[Section 3]{Levine_euler-enumerative}.
\end{deff}

With these definitions of cohomology groups, we can find relations using the ones on the graded rings (c.f. part \ref{SSEC - digression sur H}). In particular,  the functors defined in \ref{DEF - applications oubli et hyperbolique} extend uniquely to the twisted $K-$theories and so to the Rost(-Schmid) complexes, thus:

\begin{prop}\label{PROP - fonction hyperbolique bien def torsion}
For $n$ a non-negative integer, $X$ a smooth $\fF-$scheme and $\cL$ a locally free sheaf of rank 1 over $X,$ the hyperbolic application $H: \Chowct[n]{X} \to \Chowtilde[n]{X,\cL}$ of multiplication by $h$ is well defined and makes the following diagram commute:
\[\xymatrix{\Chowct[n]{X}\ar[d]_{H} \ar@{=}[r] & \Chowct[n]{X} \ar[d]^{\times 2}\\
\Chowtilde[n]{X,\cL} \ar[r]^F & \Chowct[n]{X}. }\]
\end{prop}

\begin{proof}
It is a direct application of Definition \ref{DEF - applications oubli et hyperbolique} to the Rost-Schmid complexes and the simplification of the complex for the Milnor $K-$theory ring as stated in Remark \ref{RMQ - complexe Rost-Schmid est Rost pour Milnor}.
\end{proof}

\subsection{Chow-Witt Euler classes computation}\label{SSEC - calcul classes Euler}

In the case where $\cE$ is a locally free sheaf of rank $n = \dim X$ over $X$, the Euler class $e^{CW}(\cE)$ (c.f. Definition \ref{DEF - Euler class with coefficients} where we take the Milnor-Witt spectrum as coefficients) is in general hard to compute. We will use the relations and commutative diagram \ref{PROP - fonction hyperbolique bien def torsion} to describe $\Chowtilde[n]{X,\cL}$ as a fiber product. One can complete the above construction \ref{SSEC - digression sur H} in the following cohomological statement:

\begin{align}\label{EQU - diagramme clef}
    \xymatrix{& \Chowct[n]{X}\ar[d]_{H} \ar@{=}[r] & \Chowct[n]{X} \ar[d]^{\times 2}& \\
H^n(X,\fId^{n+1}(\cL)) \ar[r] \ar@{=}[d] & \Chowtilde[n]{X,\cL} \ar[r] \ar[d]^{\mod h} & \Chowct[n]{X} \ar[r]^{\hspace{-0.5cm} \partial} \ar[d]^{\mod 2} & H^{n+1}(X,\fId^{n+1}(\cL)) \ar@{=}[d]\\
H^n(X,\fId^{n+1}(\cL)) \ar[r]_\eta & H^n(X,\fId^{n}(\cL))\ar[d] \ar[r]_\rho & \Chd[n]{X} \ar[r]^{\hspace{-0.5cm} \beta} \ar[rd]_{\Sq^2} \ar[d] & H^{n+1}(X,\fId^{n+1}(\cL)) \ar[d]^\rho\\
& 0&0& \Chd[n+1]{X}}
\end{align}
where the vertical map $\Chowtilde[n]{X,\cL} \to H^n(X,\fId^{n}(\cL))$ is induced by the quotient by $(h)$ and $\rho$ is induced by the quotient $\fId^n \to \fId^n/\fId^{n+1}.$ On the right, $\partial$ and $\beta$ are the connecting homomorphisms of the long exact sequences deduced from
\[0 \to \fId^{n+1} \xrightarrow{\eta} K^{MW}_n \to K^M_n \to 0\]
and
\[0 \to \fId^{n+1} \xrightarrow{\eta} \fId^n \to \fId^n/\fId^{n+1} \to 0.\]
The Steenrod square $\Sq^2$ is then the composition $\rho \beta.$ A construction of this diagram and the study of the morphisms can be found in \cite{HW17}.

\begin{prop}\label{PROP - CW comme image H et Witt}
Let $X$ be a smooth $n-$dimensional scheme and $\cL$ an invertible $\cO_X-$module such that
\begin{itemize}
    \item $\Chowct[n]{X}$ has no non-trivial $2-$torsion, or
    \item $H^n(X,\fId^{n+1}(\cL)) \xrightarrow{\gamma_\eta} H^n(X,\fId^{n}(\cL))$ is injective.
\end{itemize} A class in $\Chowtilde[n]{X,\cL}$ is the data of its image in $H^n(X,\underline{K}^W_n(\cL))$ and of an element  entirely determined by the value of the induced class in $\Chowct[n]{X}.$ More precisely, we have
\[\Chowtilde[n]{X,\cL} \simeq H^n(X,\sKW_n(\cL))\times_{\Chd[n]{X}}\Chowct[n]{X}.\]
\end{prop}

\begin{proof}
In maximal codimension, the diagram simplifies to
\[\xymatrix{& \Chowct[n]{X}\ar[d]_{H} \ar@{=}[r] & \Chowct[n]{X} \ar[d]^{\times 2}& \\
H^n(X,\fId^{n+1}(\cL)) \ar[r] \ar@{=}[d] & \Chowtilde[n]{X,\cL} \ar[r] \ar[d] & \Chowct[n]{X} \ar[r] \ar[d]^{\mod 2} & 0\\
H^n(X,\fId^{n+1}(\cL)) \ar[r]_\eta & H^n(X,\fId^{n}(\cL))\ar[d] \ar[r]_\rho & \Chd[n]{X} \ar[r] \ar[d] & 0\\
& 0&0&}\]
It is then a direct application of \cite[Proposition 2.11]{HW17} and the central square is a fibre product.
\end{proof}

\begin{rmq}
Until now, we did not look at the relative orientation. The Witt class considered in the above \ref{PROP - CW comme image H et Witt} needs the choice of an orientation. Usually, if the sheaf is relatively orientable we bypass this by taking the degree (c.f. \ref{DEF - degre quadratique}).
\end{rmq}

\begin{prop}\label{PROP - chowtilde non orientable est chow}
Let $X$ be isomorphic to the $d-$dimensional Grassmannian $\Grass(k,n)$ and $Y = \eP(\cF)\to X$ a projective bundle of rank $r$ over it. Assume either
\begin{itemize}
    \item $r$ and $n$ are odd or
    \item $r$ is even and not $0$.
\end{itemize}
If $\cL$ is an invertible $\cO_Y-$module not quadratically equivalent to $\omega_{Y/\fF}$, then
\[\Chowtilde[r+d]{Y,\cL} \isomr \Chowct[r+d]{Y} \simeq \eZ\cdot [pt],\]
where $[pt]$ is the class of a rational point.
\end{prop}

\begin{proof}
We begin by the first case: assume $r$ and $n$ are odd.

From the orientability of the Grassmannian given in Proposition \ref{PROP - orientabilite grassmann} and the determinant of a projective bundle given by formula \ref{EQU - omega fibre projectif}, we are in cases $\cL \simquad \cO_Y(1)\otimes p^\ast \cL'$ or $\cL \simquad p^\ast \cO_X$.

In the first case then by theorem 9.2 in \cite{Fasel_proj-bundle} we have $H^{r+d}(Y,\sKW_5(\cL)) \simeq \Chd[d]{X}$ which is $\eZ/2\eZ$. Indeed, the reduced cohomology is trivial and the injection is then the claimed isomorphism.

For the second case theorem 9.4 in \cite{Fasel_proj-bundle} where we look in codimension $r+d$, we get 
\[\tilde{H}^{r+d}(Y,\sKW_{r+d}(\cL)) \simeq H^d(X,\sKW_d) \simeq \eZ/2\eZ.\]
The last isomorphism is due to \cite[Remark 2.5]{Wendt_oriented-schubert}. The reduced cohomology is also the total cohomology in this case (just by definition).

In the end, in the non orientable case we have the fibre product (c.f. Proposition \ref{PROP - CW comme image H et Witt})
\[\xymatrix{\Chowtilde[r+d]{Y,\cL} \ar[rr] \ar[d]&& \Chowct[r+d]{Y} \simeq \eZ\cdot [pt] \ar[d] \\ 
\eZ/2\eZ \ar@{=}[rr] &&\eZ/2\eZ.}\]
Thus the desired isomorphism.

Now the second case.

The idea is the same, but the result needed to describe the cohomology is \cite[Theorem 9.1]{Fasel_proj-bundle}. From formula \ref{EQU - omega fibre projectif}, we get that $\omega_{Y/X} \simquad \cO_Y(1)\otimes p^\ast\cL'.$ 

Thus, if $\cL$ is quadratically equivalent to some $p^\ast \cL',$ the reduced cohomology $\tilde{H}^{r+d}(Y,\sKW_{r+d}(\cL))$ is isomorphic to $H^{r+d}(X,\sKW_{r+d}(\cL'))$ which is trivial. Hence, from the definition of the reduced cohomology, $H^{r+d}(Y,\sKW_{r+d}(\cL)) \simeq \Chd[d]{X}.$

Otherwise, $\cL \simquad \omega_{Y/X}\otimes p^\ast \cL'$ with $\cL'$ not quadratically equivalent to $\omega_{X/\fF}.$ Then, the second isomorphism of \cite[Theorem 9.1]{Fasel_proj-bundle} gives $\tilde{H}^{r+d}(Y,\sKW_{r+d}(\cL)) \simeq H^{d}(X,\sKW_{d}(\cL'))$ and from \cite[Remark 2.5]{Wendt_oriented-schubert}, this last one is $\eZ/2\eZ$. Moreover, the reduced cohomology is the whole cohomology as the cokernel of a trivial map.

The two possibilities merge into the same diagram as before and we get the fibre product
\[\xymatrix{\Chowtilde[r+d]{Y,\cL} \ar[rr] \ar[d]&& \Chowct[r+d]{Y} \simeq \eZ\cdot [pt] \ar[d] \\ 
\eZ/2\eZ \ar@{=}[rr] &&\eZ/2\eZ.}\]
\end{proof}

\begin{tho}[Non orientable Euler class]\label{THO - classe Euler non orientable}
Let $X$ be the $d-$dimensional Grassmannian $\Grass(k,n)$. Assume either
\begin{itemize}
    \item $n$ is odd and $Y = \eP(\cF)$ is a projective bundle of odd rank $r$ over $X$ or
    \item $Y = \eP(\cF)$ is a projective bundle of even rank $r$ over $X$.
\end{itemize}
Let $\cE$ be a coherent sheaf, locally free of rank $d+r$ over $Y$. If $\cE$ is not orientable over the base field $\fF$ then $e^{CW}(\cE)$ is uniquely determined by its value in the Chow group.

Furthermore, if $\deg c_{r+d}(\cE)$ is even, one can write 
\[e^{CW}(\cE) = \frac{\deg c_{r+d}(\cE)}{2}h \in \Chowtilde[r+d]{Y,\det \cE}\simeq \Chowct[r+d]{Y}.\]

We also have $e^W(\cE) = 0 \in H^{r+d}(Y,\sKW_{r+d}(\det \cE))$ in this case.
\end{tho}

\begin{rmq}
The main difference with Proposition \ref{PROP - chowtilde non orientable est chow} is the possibility to explicitly describe an element of even rank in the Chow-Witt group twisted by a non orientable invertible $\cO_Y-$module.
\end{rmq}

\begin{proof}
Using Proposition \ref{PROP - chowtilde non orientable est chow}, what is left to see is how to get the preimage under the isomorphism. Under the hypothesis of even degree, it is just an application of Corollary \ref{PROP - fonction hyperbolique bien def torsion}.
\end{proof}

\textbf{Warning:}\ An important difference between what was said until here and the next statements of this section is the orientability. Until here, results were mostly to have concrete descriptions only when $\cL \not \simquad \omega_X$. Some description can be given when we have orientability, but one has to be careful in which cohomology things happen and we have to use the quadratic degree.

\begin{cor}\label{COR - groupe Chow-Witt corps}
If $X = \Spec{\fF},$ we have obvious isomorphisms (using the previous notations and definitions introduced in this section \ref{SEC - calcul classe Euler}):
\begin{itemize}
\item $H^0(X,\sKW_0) \simeq \bfW^0(X) \simeq \Witt(\fF),$
\item $\Chowct[0]{X} \simeq \eZ,$
\item $\Chowtilde[0]{X} \simeq \GW(\fF)$.
\end{itemize}
In particular, the Grothendieck-Witt degree can be determined through the decomposition into the Witt degree and Chow degree. The difference is then a multiple of $h$.
\end{cor}

\begin{proof}
The last statement is an application of the previous Proposition \ref{PROP - CW comme image H et Witt} where we take $X = \Spec{\fF}$ and $\cL = \cO_X$. All isomorphisms are just the relations between the $K-$theory groups in degree $0$.
\end{proof}

\begin{rmq}
The first isomorphism in Corollary \ref{COR - groupe Chow-Witt corps} is the main reason why it is enough to compute Euler classes with Witt coefficients when we take the degree in the end.
\end{rmq}

\begin{rmq}
One of the reasons why we end up only computing the degree is that the fibre product
\[\xymatrix{\Chowtilde[n]{X,\omega_X} \ar[r] \ar[d] & \Chowct[n]{X} \ar[d]^{\mod 2}\\
 H^n(X,\fId^{n}(\omega_X))\ar[r]_\rho & \Chd[n]{X}}\]
 and especially its bottom part is quite unknown since $H^n(X,\fId^{n}(\omega_X))$ has no explicit and usable description in general. In particular, the commutative diagram of multiplication by $h$ (c.f. Proposition \ref{PROP - fonction hyperbolique bien def torsion}) is not enough if the Chern degree is even because the top map has no reason to be injective a priori.
\end{rmq}

In general, even the Witt degree is hard to compute, but if the sheaf is of odd rank, we don't need computations:

\begin{lemme}[Lemma 4.3, \cite{Levine_euler-enumerative}]\label{LEM - annulation classe Witt rang impair}
Let $\cE$ be a locally free sheaf of odd rank $r$ on $Y$ a smooth $\fF-$scheme. Then,
\[\eta e^{CW}(\cE) = 0\]
in $H^r(Y,\underline{K}^{MW}_{r-1}(\det \cE)).$
Moreover, its specialisation in $\bfW^5(Y,\det \cE)$ is also $0$. 
\end{lemme}

\begin{prop}[Orientable Euler number of odd rank]\label{PROP - classe Euler orientable rang impair}
Assume $\cE$ is an locally free orientable sheaf of odd rank $r$ over a smooth and proper $\fF-$scheme $X$ of dimension $r$. Then 
\[\deg^{CW} e^{CW}(\cE) = \frac{\deg c_r(\cE)}{2}h\in \GW(\fF)\]
\end{prop}

\begin{proof}
If we note $f: X \to \Spec{\fF}$ the structural morphism, then we look at the value of $f_\ast e^{CW}(\cE)$ in $\Chowtildecov[0]{\Spec\fF} = \Chowtilde[0]{\Spec\fF}.$ Then, the value is determined by the value in $H^0(X,\sKW_0) \simeq \bfW^0(X)$ and in $\Chowct[0]{\Spec\fF}$. But the first one is 0 by Lemma \ref{LEM - annulation classe Witt rang impair} and by the fibre product property \ref{PROP - CW comme image H et Witt}, we have the rest.
\end{proof}

\begin{rmq}\label{RMQ - donnee Witt est la seule importante}
In the end, when computing Euler classes (or degree), the "quadratic information" is solely contained in the cohomology with coefficients in $K^W$ (or the Witt group when orientable). The other part is then just a translation of the "classical" result over an algebraically closed field back in the Chow-Witt group (or Grothendieck-Witt group).
\end{rmq}

\begin{rmq}
A subsidiary question is to what extent this Witt invariant keeps track of the geometry and is not a Witt invariant of a finite length étale algebra over the base field. Serre described these invariants in \cite{Serre_Witt-inv}. More recently, in \cite{BRW25} a link was found for some Welschinger--Witt invariants.
\end{rmq}

\section{Lines in a degree 2 del Pezzo}

\subsection{Del Pezzo surfaces of degree 2}\label{SSEC - surfaces de degre 2}

We will imitate the strategy at the beginning of \cite{Tih_1980} but adapted to a surface.

\begin{prop}\label{PROP - surface deg 2 dans fibre projectif}
A smooth del Pezzo surface $S$ of degree $2$ is a double cover of $\eP^2_\fF$ ramified over a smooth quartic. This can be extended to an embedding of $S$ in the projective bundle $P = \eP(\cO_{\eP^2} \oplus \cO_{\eP^2}(2))$ as an hypersurface defined by the zero locus of a section of $\cO_P(2)$. This gives the commutative diagram 
\[\xymatrix{S\ \ar@{^{(}->}[r] \ar[rd]_w & P \ar[d]^{w_0}\\ & \eP^2 }\]
where $w$ is the anti-canonical morphism.
\end{prop}

\begin{proof}
The first part comes from the definition of a smooth del Pezzo surface of degree two and its ample anti-canonical sheaf. A characterisation of the degree is $\dim H^0(S,\omega_S^\vee) - 1 = \deg S$ and since $\omega_S^\vee$ is spanned by its global sections the evaluation $H^0(S,\omega_S^\vee)\otimes \cO_S \to \omega_S^\vee$ is surjective. Thus the projection $w$ to $\eP^2 = \eP(H^0(S,\omega_S^\vee))$ defines the double cover along a smooth quartic (c.f. \cite[Section 8.7]{Dolgachev}). In particular, $w^\ast \cO_{\eP^2}(1) = \omega_S^\vee.$

Let us consider the sheaf $\cE = \cO_{\eP^2}\oplus \cO_{\eP^2}(2).$ Then $w^\ast\cE = \cO_S \oplus (\omega_S^\vee)^{\otimes 2}$ and the second projection $p_2 : w^\ast\cE \to (\omega_S^\vee)^{\otimes 2}$ determines a $\eP^2-$morphism $j : S \to \eP(\cE).$ This morphism $j$ is a closed immersion since the induced morphisms $\Symalg[r]{w^\ast \cE} \to (\omega_S^\vee)^{\otimes 2r}$ are surjective. One can also remark that $\eP(\cE) \to \eP(1,1,1,2)$ is the blow-up of the singular point in $\eP(1,1,1,2)$ and that $j(S)$ is disjoint from the exceptional divisor corresponding to the section we get from the first projection $\cO_{\eP^2}\oplus \cO_{\eP^2}(2) \to \cO_{\eP^2}.$

What is left to see is that $\cO_{P}(S) \simeq \cO_P(2).$ But we know that $w$ is a double cover the the closed immersion $j$ gives $j(S)$ as a 2-section of $w_0 : P \to \eP^2$. Hence $\cO_P(S) \simeq \cO_P(2)\otimes w_0^\ast \cO_{\eP^2}(d)$ for some $d.$ From the isomorphism \ref{EQU - omega fibre projectif}, we get
\[\omega_{P} \simeq \cO_{P}(-2) \otimes w_0^\ast \cO_{\eP^2}(2) \otimes w_0^\ast \cO_{\eP^2}(-3) \simeq \cO_P(-2)\otimes w_0^\ast \cO_{\eP^2}(-1). \]
By adjunction formula 
\[\omega_S \simeq \rest{\omega_{\eP(\cE)}}{S} \otimes \det \cC_{S/\eP(\cE)}^\vee,\]
where $\cC_{S/\eP(\cE)}^\vee \simeq \rest{\cO_{\eP(\cE)}(S)}{S}$ is the normal sheaf to $S$ in $\eP(\cE),$ we then get 
\begin{align*}
     w^\ast \cO_{\eP^2}(-1) = \omega_S & \simeq j^\ast\left(\cO_P(-2)\otimes w_0^\ast \cO_{\eP^2}(-1)\otimes \cO_P(2) \otimes w_0^\ast \cO_{\eP^2}(d)\right)\\ 
     & \simeq j^\ast w_0^\ast \cO_{\eP^2}(d-1).
\end{align*}
Thus $d$ is zero and the claim.
\end{proof}

\subsection{Lines in a projective bundle}

The lines in $P$ we are looking for are generalisation of lines in $S$. So we need to first define what we want as lines of $S.$

\begin{deff}[Hilbert scheme of lines]\label{DEF - schema Hilbert droites}
The Hilbert scheme of lines of a del Pezzo surface $X,$ with respect to the ample sheaf $\omega_X^\vee,$ is the scheme
\[\Hdrtes X = \Hilbsch[1+t]{X}\]
representing the functor $h_{1+t}^{\omega_X^\vee}(X)$ of coherent sheaves on $X$ such that their Hilbert polynomial (c.f. \cite[Section 5.1.4]{FGA_explained}), with respect to $\omega_X^\vee,$ is $1+t.$
\end{deff}

For $S$ a del Pezzo surface, what we take for a line is then a 1-dimensional subvariety such that its ideal sheaf has Hilbert polynomial $1+t$ with respect to the anti-canonical sheaf $\omega_S^\vee.$ We have to make the choice of the ample sheaf $\cL$ of the definition and the polynomial to match this and to be coherent with what we want for lines in $P$. Explicitly, we want an ample sheaf $\cL$ on $P$ and a Hilbert polynomial $\phi$ such that $\Hilbsch[j^\ast\cL,\phi]{S}\hookrightarrow \Hilbsch[\cL,\phi]{P}$ and $\Hilbsch[j^\ast\cL,\phi]{S}$ isomorphic to $\Hilbsch[\omega_S^\vee,1+t]{S}.$

\begin{lemme}\label{LEM - droites P comme zero sections}
If $\fF = \bar{\fF}$ is algebraically closed, the lines of $S$ are $(-1)-$curves that project to lines bi-tangent to the quartic in $\eP^2.$ With respect to the embedding in $P$ they are described as zero loci 
\begin{align}\label{EQU - droites de P}
    Z_P(s),\quad s \in H^0(P,\cO_P(1)\oplus w_0^\ast \cO_{\eP^2}(1)).
\end{align}
\end{lemme}

\begin{proof}
As explained before, we have to choose some $\cL$ over $P.$ If we take $\cL = \cO_P(1)\otimes w_0^\ast\cO_{\eP^2}(1),$ then $j^\ast \cL = (\omega_S^\vee)^{\otimes 3}$ by construction. Thus a line in $S$ is a curve $l$ such that $\deg \rest{j^\ast\cL}{l} = 3.$ The Hilbert polynomial, with respect to such sheaf, is
\[P_l(t) = \sum_i (-1)^i \dim_\fF (H^i(l,(\rest{j^\ast\cL}{l})^{\otimes t})) = \dim H^0(\eP^1,\cO_{\eP^1}(3t)) = 1+3t.\]

On the other hand, the image of a line $l$ in $\eP^2$ is $w(l) = l_0$ tangent to the ramification quartic in $\eP^2$. If that were not to be the case, $w^{-1}(l_0)$ would be an irreducible curve $c$ in $S$ with intersection $c\cdot K_S = 2$ as $w$ is a double cover. Hence, $l_0$ is bi-tangent to the quartic. Then, the image $j(l)$ is a curve $C$ in $\rest{P}{l_0} = \eP(\cO_{l_0}\oplus \cO_{l_0}(2))$ which is the zero locus of a global section of $\rest{\cO_P(1)}{l_0}.$ We get the wanted description of lines as in \ref{EQU - droites de P} but we have to verify that the choice of $\cL$ gives the same Hilbert polynomial on both sides.

Let's assume $C$ is isomorphic to $\eP^1,$ then $\rest{\cL}{C} = \rest{\cO_P(1)}{C}\otimes\rest{w_0^\ast \cO_{\eP^2}(1)}{C}$ is isomorphic to $\cO_{\eP^1}(2)\otimes\cO_{\eP^1}(1) = \cO_{\eP^1}(3).$ Thus, its Hilbert polynomial is $P_C(t) = 1+3t = P_l(t)$.

Then the choices coincide, we have the closed immersion $\Hilbsch[(\omega_S^\vee)^{\otimes 3},1+3t]{S}\hookrightarrow \Hilbsch[\cO_P(1)\otimes w_0^\ast\cO_{\eP^2}(1),1+3t]{P}$ and the isomorphism $\Hilbsch[(\omega_S^\vee)^{\otimes 3},1+3t]{S} \simeq \Hilbsch[\omega_S^\vee,1+t]{S}.$
\end{proof}

\begin{para}
The above consideration generalises easily to the following (and this will be of use): 

Let us fix $\eP^n = \eP(V)$ over $\fF$ and $X = \eP_{\eP^n}(\cE)$ with $\cE = \cO_{\eP^n} \oplus \cO_{\eP^n}(d).$ Note $p: X \to \eP^n.$
We will look at sub-schemes of $X$ defined as
\begin{align}\label{EQU - section relatives a hyperplane}
    Z(s),\quad s \in H^0(X,\cO_X(1)\oplus p^\ast \cO_{\eP^n}(1))
\end{align}
i.e. zero loci of sections of $\cO_X(1)$ restricted to a hyperplane of $\eP^n$. We will refer to it by $(n-1)-$planes.
\end{para}

\begin{prop}[1.1 in \cite{Tih_1980}]\label{PROP - schema Hilbert droites de P}
Let $Y$ be a base of the family of $(n-1)-$planes of $X$. Then $Y$ is isomorphic to $\eP(\cO_G \oplus \Symalg[d]{\cQ^\vee})$ with $G = \Grass(n+1,n)$ the Grassmannian of hyperplanes in $\eP^n$ and $\cQ$ the universal quotient of $G$. We will also note $q$ the natural projection from $Y$ to $G.$
\end{prop}

\begin{rmq}
In our specific case, we get $Y = \eP(\cO_G \oplus \Symalg[2]{\cQ^\vee})$ over $\Grass(3,2).$
\end{rmq}

We will first prove some lemmas.

\begin{lemme}\label{LEM - schema Hilbert diviseurs}
The Hilbert scheme $\Sigma X$ that parameters the zero loci of non-zero sections of $\cO_X(1)$ is isomorphic to the projective space
\[\eP(H^0(X,\cO_X(1))^\vee) \simeq \eP(\Symalg[d]{V^\vee}\oplus \fF).\]
\end{lemme}

\begin{proof}
Since the schemes at hand represent functors of points, we will check the statement at the level of points.
Let $f : T \to S = \Spec{\fF}$ be a morphism. A $T-$point of a $S-$scheme is a point of $X\times_S T = X_T$, i.e. a morphism $T \to X_T.$

A $T-$point of this projective space is a $1-$dimensional quotient $H^0(X,\cO_X(1))_T^\vee \to L_T$. It is the same as a $1-$dimensional sub-space $L_T^\vee \subseteq H^0(X,\cO_X(1))_T$ by duality. Then, choosing a generator $s$ in $L_T^\vee$ is the same as choosing a non-zero global section of $\cO_X(1).$ Changing $s$ changes the section but not its zero locus so the only important datum is $L_T^\vee.$
\end{proof}

\begin{lemme}\label{LEM - donnee droite comme section et hyperplane}
A line in $X$ is given by the choice of $H$ an hyperplane of $\eP^n$ and of a non-zero section of $\rest{\cO_X(1)}{p^{-1}(H)}.$
\end{lemme}

\begin{rmq}
This lemma is crucial in that it permits to look at the problem "point-wise", with respect to points of the Grassmannian (or hyperplanes of $\eP^n$ equivalently).
\end{rmq}

\begin{proof}
We want to describe $(n-1)-$planes of $X$ as they are defined in \ref{EQU - section relatives a hyperplane}. So, we want the scheme parameterizing the zero loci of sections of $\cO_X(1)\oplus p^\ast \cO_{\eP^n}(1)$ over $X,$ i.e. non-zero sections of $\rest{\cO_X(1)}{p^{-1}(l)}$ where $H$ is a hyperplane of $\eP^n.$ Moreover, the restriction 
\[H^0(X,\cO_X(1)) \to H^0(p^{-1}(H),\rest{\cO_X(1)}{p^{-1}(H)})\]
is surjective. Thus, the datum of a $(n-1)-$plane is that of a hyperplane in $\eP^n$ and a non-zero section of $\rest{\cO_X(1)}{p^{-1}(H)}.$
\end{proof}

\begin{proof}[Proof (of Proposition \ref{PROP - schema Hilbert droites de P})]
From Lemma \ref{LEM - donnee droite comme section et hyperplane}, we can first look at global sections 
\[H^0(X,\cO_X(1)) \simeq H^0(\eP(V),\cE) \simeq \Symalg[d]{V}\oplus \fF.\]
Restricting to a $T-$hyperplane, i.e. a $n-$dimensional quotient $\fq_T$ of $V_T$ which is in fact the choice of a $T-$point in $\Grass(V,n)$ gives the space $\Symalg[d]{\fq_T}\oplus \fF_T$ and in fine $\Symalg[d]{\fq_T^\vee}\oplus \fF_T$ since we look at zero loci. Thus, let us describe $Y$ as a projective bundle $\eP(\cF)$ over the Grassmannian $\Grass(V,n)$, with $\cF = \Symalg[d]{\cQ^\vee}\oplus \cO_G$ and $\cQ$ the universal quotient. Since all the objects represent functors. It is enough to verify the identity at the level of $T-$points.

Let $t$ be a $T-$point of $\eP(\cF).$ We then have the data of a quotient of rank $n,\ V \to W$ corresponding to $q(t)$ and a rank $1$ quotient of $\rest{\Symalg[d]{\cQ^\vee}\oplus \cO_G}{q(t)}$. From the universal property, $\rest{\cQ}{q(t)} = W$, hence a rank $1$ quotient $\Symalg[d]{W^\vee} \oplus \fF_T \to L$. It gives rise to a hyperplane $H = \eP(W)$ and the quotient $L$ determines a unique divisor of $\eP(\rest{\cE}{H})$ by using the Lemma \ref{LEM - schema Hilbert diviseurs} on $\eP(W).$

On the other hand, from a $T-$scheme $C = Z(s_1)\cap p^{-1}(H)$, with $s_1$ a global section of $\cO_X(1)$ and $H$ a hyperplane of $\eP(V_T),$ we have to determine a point of $\eP(\cF).$ The hyperplane $H$ corresponds to a quotient $W$ of dimension $n$ in $V_T$. Thus, we have a $T-$point of the Grassmannian $\Grass(V,n)$ such that $\cQ_T = W.$ Moreover, $C$ is equal to $Z\left(\rest{s_1}{\eP\left(\rest{\cE}{l}\right)}\right)$ and with Lemma \ref{LEM - schema Hilbert diviseurs}, this zero locus is exactly a point of $\eP(\Symalg[d]{W^\vee}\oplus\fF_T).$
\end{proof}

Up to this point, we have the lines of $P.$ We now want to describe the lines in $S,$ that is of a section of $\cO_P(2),$ in $Y.$ To pass this inclusion in $P$ to something on $Y,$ we will follow once again \cite{Tih_1980}. The bridge between the two will be achieved through some incidence schemes.

\begin{lemme}[Incidence scheme between the Grassmannian and $\eP^n$]\label{LEM - schema incidence grassmannienne}
Let $S$ be a base scheme and $\cV$ a locally free sheaf of rank $n+1$ on $S$. There exists an incidence scheme $I$ between points of $G = \Grassm[r]{\cV}$ and projective sub-bundles of dimension $r-1$ of $\eP_S(\cV).$ Moreover, if by $\cQ$ we denote the universal quotient of $G$ then $I \simeq \eP(\cQ)$ over $G$. If $r = 2,$ we also have $I \simeq\eP(\Omega(1))$ over $\eP_S(\cV).$
\end{lemme}

\begin{proof}
We will build $I$ as $\Grassm[r-1]{\Omega_{\eP_S(\cV)}(1)}$ over $\eP_S(\cV)$ and then verify that it has a morphism to $\Grassm[r]{\cV}$ which factorises through $\eP(\cQ)$ into an isomorphism $i$. The Euler sequence of $\Grassm[r-1]{\Omega_{\eP_S(\cV)}(1)}$ (the vertical row on the left) and the one of $\eP_S(\cV)$ (the middle horizontal line) give the following diagram
\[\xymatrix{\cS_1 \ar@{=}[r] \ar[d] &\cS_1  \ar[d] \ar[r] & 0\ar[d]\\
p^\ast\Omega_{\eP_S(\cV)}(1) \ar[r] \ar[d] & p^\ast(\cV\otimes \cO_{\eP_S(\cV)}) \ar[r]\ar[d] & p^\ast\cO_{\eP_S(\cV)}(1) \ar[d] \\
\cQ_1 \ar[r] & \cQ_r \ar[r]  & \cL_1  }\]
where the Euler exact sequence of $\eP_S(\cV)$ was pulled back by $p : \Grassm[r-1]{\Omega_{\eP_S(\cV)}(1)} \to \eP_S(\cV).$
The induced locally free quotient sheaf $\cQ_r$ of rank $r$ of $\cV\otimes\cO_{\Grassm[r-1]{\Omega_{\eP_S(\cV)}(1)}}$ gives a morphism to $\Grassm[r]{\cV}.$ Let's call this morphism $f$ and then $\cQ_r = f^\ast\cQ$ with $\cQ$ the universal quotient of $\cV\otimes \cO_{\Grassm[r]{\cV}}.$

To factorise $f$ through $q : \eP(\cQ) \to \Grassm[r]{\cV},$ we need a rank one quotient of $f^\ast\cQ.$ We complete the previous diagram, and we have it as the cokernel $\cL_1$ of $\cQ_1 \to f^\ast \cQ,$ which is isomorphic to $p^\ast \cO_{\eP_S(\cV)}(1).$ So we have a morphism $i : \Grassm[r-1]{\Omega_{\eP_S(\cV)}(1)} \to \eP(\cQ)$ and $p^\ast\cO_{\eP_S(\cV)}(1) = i^\ast\cO_{\eP(\cQ)}(1).$ Since all objects represent functors and that $i$ is an isomorphism at the level of points (just by base-changing to points the above diagram), $i$ is an isomorphism of schemes. 
\end{proof}

\begin{rmq}
From the proof, we have in particular that $p^\ast \cO_{\eP_S(\cV)}(1) = \cO_{\eP(\cQ)}(1)$ when identifying $I$ as $\eP(\cQ).$
\end{rmq}

\begin{prop}\label{PROP - dimension du fibre vectoriel voulu}
Let's write $\fU \subset X\times Y$ the incidence scheme, following from the universal property of the Hilbert scheme $Y.$ It has two natural projections $\pi_1: \fU \to X$ et $\pi_2: \fU \to Y$. Then, the sheaf $\pi_{2,\ast}\pi_1^\ast \cO_X(m)$ is locally free of rank $\dim H^0(\eP^{n-1},\cO_{\eP^{n-1}}(dm))$ on $Y$.
\end{prop}

\begin{lemme}\label{LEM - diviseurs reductibles de O_X1}
Note $E_0$ the divisor given by the projection $s_0: \cO_{\eP^n}(d)\oplus \cO_{\eP^n} \to \cO.$

If the zero locus of a section of $\cO_X(1)$ is reducible, it is of the form $E_0 \cup p^{-1}(Z_d)$ with $Z_d$ a degree $d$ divisor in $\eP(V)$. 
\end{lemme}

\begin{proof}
Observe that $p_\ast \cO_X(1) = \cE = \cO_{\eP^n}(d)\oplus \cO_{\eP^n}$ and reducible divisors come from sections $\cO_{\eP^n} \to p_\ast\cO_X(1)$ such that the composition with $s_0$ is 0. So, these sections define an element of $\cO_{\eP^n}(d)$, i.e. a degree $d$ divisor on $\eP^n.$
\end{proof}

\begin{lemme}\label{LEM - incidence locale}
The incidence scheme $\fD$ between $X$ and the Hilbert scheme $\Sigma X$ (c.f. \ref{LEM - schema Hilbert diviseurs}) is closed in $X\times \Sigma X.$ Furthermore, it is birational to the projective bundle $\eP_{\Sigma X}(V \otimes \cO_{\Sigma X})$.
\end{lemme}

\begin{proof}\label{DEM - incidence locale}
It is clear that $\fD$ is closed in $X\times \Sigma X$ and we can write the diagram
\[\xymatrix{\fD \ar[rr]^{\varphi = p\times \operatorname{id}\quad } \ar[d]_{p_2} && \eP(V) \times \Sigma X \ar[d]^{\operatorname{pr}_2} \\\Sigma X \ar@{=}[rr]&& \Sigma X.}\]
We will show that $\varphi$ is a birational morphism of $(\Sigma X)-$schemes.

From Lemma \ref{LEM - schema Hilbert diviseurs}, we get $\Sigma X \simeq \eP(\Symalg[d]{V^\vee} \oplus \fF)$. In this projective space, we have a preferred hyperplane $\fH_\infty = \eP(\Symalg[d]{V^\vee})$ corresponding to the projection $\Symalg[d]{V^\vee}\oplus \fF \to \Symalg[d]{V^\vee}$.

To understand $\varphi,$ we have to look at its action on $(n-1)-$planes, i.e. the projection of zero-locus of non-zero sections of $\cO_X(1)$ to $\eP(V).$ We either get an irreducible divisor $D \subseteq X$ and an isomorphism $\rest{p}{D}: D \to \eP(V)$ or a reducible divisor of the form $E_0 \cup p^{-1}(Z_d)$ (c.f. Lemma \ref{LEM - diviseurs reductibles de O_X1}). The degree $d$ divisors in $\eP(V)$ have their zero loci precisely described by $\eP(\Symalg[d]{V^\vee}).$ Thus the reducible divisors correspond to points in $\fH_\infty.$ This describes a blow-up of centre the universal element in $\eP(V) \times \fH_\infty$ corresponding to degree $d$ divisors of $\eP(V).$
\end{proof}

\begin{proof}[Proof (of Proposition \ref{PROP - dimension du fibre vectoriel voulu})]
We have $\fU$ and its two projections $\pi_1, \pi_2.$ We can compose the inclusion $\fU \hookrightarrow X \times Y$ by the projection $p\times q: X \times Y \to \eP(V) \times \Grass(V,n).$ This projection tautologically factorise through the incidence scheme of hyperplanes in $\eP(V),\ I$ (c.f. Lemma \ref{LEM - schema incidence grassmannienne}). In particular, we get an induced morphism 
\[\Phi: \fU \to I\times_{\Grass(V,n)} Y \simeq \eP_I(q_2^\ast \cF)\]
noting $q_2$ the projection $I \to \Grass(V,n).$
Let us fix a $(T-)$point of the Grassmannian, i.e. a quotient $V_T \to W$ of dimension $n.$ Lemma \ref{LEM - incidence locale} applied on $\eP(W)$ together with Lemma \ref{LEM - donnee droite comme section et hyperplane} give that $\Phi$ is birational and induces an isomorphism over $\eP_I(q_2^\ast \cF) \setminus \eP_I(q_2^\ast \Symalg[d]{Q^\vee})$.

Let $F = \pi_2^{-1}(s)$ be a fiber over some $s$ a $T-$point in $Y$ and $H = \eP(W) = \eP^{n-1}$ be the corresponding hyperplane to $q(s).$ Everything is then happening over $H$. So, we can decompose between reducible and irreducible divisors of $\cO_{\eP(\rest{\cE}{H})}(1)$ as done in the proof of Lemma \ref{LEM - incidence locale}. Indeed, in the universal family, $F$ is a divisor as the zero locus in $\rest{X}{H}$ of a non-zero global section of $\cO_{\eP(\rest{\cE}{H})}(1)$. Then, if $F$ is irreducible, it is isomorphic to $H$ through the projection. Thus, we have $\pi_1^\ast\rest{\cO_X(m)}{F} \simeq \rest{\cO_{\eP(\rest{\cE}{H})}(m)}{D} \simeq \cO_H(md).$ On the other hand, if $F$ is reducible, it is isomorphic to $\rest{E_0}{H} \cup p^{-1}(H \cap Z_d)$ with $Z_d$ a degree $d$ divisor in $H$. By nature of $E_0$, we get $\rest{\cO_X(m)}{\rest{E_0}{H}} = \cO_{\rest{E_0}{H}}$ et $\rest{\cO_X(m)}{p^{-1}(H \cap Z_d)} = \cO_{\eP(\rest{\cE}{H\cap Z_d})}(m).$ In each case, we have $H^1(F,\pi_1^\ast\rest{\cO_X(m)}{F}) = 0.$ Since the projection morphism $\pi_2$ is proper and flat, we get, by flat base change of cohomology \cite[\href{https://stacks.math.columbia.edu/tag/02KH}{Tag 02KH}]{stacks-project} (or \cite[Corollaire 6.9.10]{EGA3_2}) part 1:
\[R^1\pi_{2,\ast}\pi_1^\ast \cO_X(m) = 0.\]
So, $\pi_{2,\ast}\pi_1^\ast \cO_X(m)$ is a locally free sheaf on $Y$. Its rank is the dimension of its fibers which is, by \cite[\href{https://stacks.math.columbia.edu/tag/02KH}{Tag 02KH}]{stacks-project} part 2:
\[\dim H^0\left(F,\rest{\pi_1^\ast\cO_X(m)}{F}\right) = \dim H^0(\eP^{n-1},\cO_{\eP^{n-1}}(md)).\]
\end{proof}

\begin{cor}\label{COR - droites comme zero section fibre}
For the double cover $S$ of $\eP^2$ ramified on a smooth quartic, we get that the Hilbert scheme of lines is isomorphic to the zero locus of a global section of the locally free sheaf $\cM = \pi_{2,\ast}\pi_1^\ast \cO_P(2)$ of rank $5$ on $Y = \eP(\cO_G \oplus \Symalg[2]{\cQ^\vee})$ (a projective bundle of rank $3$ over the Grassmannian $G = \Grass(3,2)$).
\end{cor}

\begin{rmq}\label{RMQ - manque information}
Even though we lack information about $\cM$, we are in a special case were computation is still possible thanks to the results of part \ref{SSEC - calcul classes Euler}.
\end{rmq}

\subsection{Euler class of an odd rank sheaf}
Since we don't have an explicit description of the determinant, we don't know if we can take the quadratic degree. We only have $e(\cM) \in \eE^5(Y,\det \cM)$ which is isomorphic to $\eE^{BM}_0(Y/G,\det \cM \otimes \omega_{Y/G}^\vee)$ for $\eE$ any $\SL^c-$oriented theory.

But if we specialise to Milnor $K-$theory, i.e. Chow groups, it suffice to compute the Chern class when $\fF$ is algebraically closed. The degree of $c_5(\cM)$ is then the length of the Hilbert scheme of lines of a degree 2 smooth del Pezzo surface over $\bar{\fF}.$ It is known that there are 56 such lines (c.f. \cite[Section 8.7]{Dolgachev}).

\begin{lemme}\label{LEM - classe chern M sur C}
Let $\cM= \pi_{2,\ast}\pi_1^\ast\cO_{\eP^2}(2)$ be the same locally free sheaf of rank $5$ over $Y= \eP(\cO_G\oplus \Symalg[2]{\cQ^\vee})$ as before. Then the top Chern class is 
\[c_5(\cM) = 56[pt] \in \Chowct[5]{Y} \simeq \eZ\cdot [pt].\]
\end{lemme}

\begin{tho}[Main theorem]\label{THO - classe Euler del Pezzo deg 2}
Let $Y = \eP(\cO_G\oplus\Symalg[2]{\cQ^\vee})$ be the Hilbert scheme of lines of $P.$ It is a projective bundle over $G = \Grass(3,2).$ Let $\cM = \pi_{2,\ast}\pi_1^\ast\cO_{\eP^2}(2)$ be the locally free $\cO_Y-$module of rank $5$. Since we don't know $\det \cM$ we have to separate what can happen into:
if $\cM$ is not orientable, the Euler class is
\[e^{CW}(\cM) = 28h \in \Chowtilde[5]{Y,\det \cM},\] 
if $\cM$ is orientable, the degree of the Euler class is
\[\deg^{CW}e^{CW}(\cM) = 28h \in \GW(\fF).\]
\end{tho}

\begin{proof}
From Theorem \ref{THO - classe Euler non orientable}, if $\cM$ is not orientable we only have the value in the Chow group. But from Lemma \ref{LEM - classe chern M sur C}, we know it is $56[pt]$. Since $56$ is even, we can use Corollary \ref{PROP - fonction hyperbolique bien def torsion} to give a description as $28h$. In particular, its image in $H^5(Y,\sKW_5(\det \cM))$ is $0$.

The last case is when $\det \cM$ is quadratically equivalent to $\omega_{Y/\fF},$ that is when $\cM$ is relatively orientable over the base field. Then we have the result by Proposition \ref{PROP - classe Euler orientable rang impair} and still the rank by Lemma \ref{LEM - classe chern M sur C}.
\end{proof}

The really "motivic information" is contained in the Witt class (or degree if orientable) we have in the proof (c.f. remark \ref{RMQ - donnee Witt est la seule importante}).

\section{Degree 4 del Pezzo surfaces}

\subsection{Intersection of two quadrics}

\begin{prop}[\cite{Dolgachev} Theorem 8.6.2]\label{PROP - surface deg 4 comme intersection quadriques}
Let $S$ be a del Pezzo surface of degree 4 on $\fF$. Then $S$ is a complete intersection of two quadrics in $\eP^4_\fF$.
\end{prop}

A priori, only the intersection has to be smooth but not the two quadric hypersurfaces. To justify that we can choose them to be smooth, we follow what was done in \cite[Section 3]{KST_DP4}. The surface $S$ is smooth and defined by two quadrics $q_0,\ q_1.$ We can take the pencil $\bfQ$ parameterised by a $\eP^1_\fF$ to try and change them for smooth representatives. 

The singular quadrics in the pencil are given by the roots of $\det(\bfQ)$. From \cite[Proposition 2.1]{Reid_thesis} we know that this polynomial has exactly 5 roots over the algebraic closure. We always have enough points on the $\eP^1_\fF$ parameterizing the pencil of quadrics to avoid roots except over $\eF_3$ or $\eF_5$ eventually.

Thus we can state:
\begin{prop}
For $\fF$ of cardinal at least 6, a smooth del Pezzo surface $S$ of degree $4$ on $\fF$ is the complete intersection of two smooth quadrics in $\eP^4_\fF$.
\end{prop}

\subsection{Euler class computation}\label{SSEC - calcul degre 4}

We end up in a very similar configuration to the one of the cubic surface in \cite{Kass_Wickelgren}. 

\begin{prop}\label{PROP - lines of degree 4 del Pezzo}
The lines of a del Pezzo surface $S$ of degree 4, over $\fF$ of cardinal at least 6, are a subscheme $\Hilbsch[1+t]{S}$ of dimension 0 in $G = \Grass(5,2)$ characterised as
\[\Hilbsch[1+t]{S} = Z(s_1\oplus s_2) \subseteq \Grass(5,2),\ s_1\oplus s_2 \in H^0(G,\Symalg[2]{\cQ}\oplus \Symalg[2]{\cQ}).\]
\end{prop}

\begin{proof}
The two quadrics define two elements in $\cO_{\eP^4}(2).$ Through the push-pull formula in the incidence scheme \ref{LEM - schema incidence grassmannienne}, we get a global section of $\Symalg[2]{\cQ}\oplus \Symalg[2]{\cQ}$ on the Grassmannian. Hence the description of the Hilbert scheme of lines as zero locus of a global section.
\end{proof}

Compared to the degree 2 case, we have a complete description of the sheaf here. We can check if it is relatively orientable or not. We have $\det(\Symalg[2]{\cQ}\oplus \Symalg[2]{\cQ}) \simeq \cO_G(3)^{\otimes 2} \sim \cO_G$ on one hand. On the other side $\omega_G \simeq \cO_G(-5) \sim \cO_G(1)$. Thus, the locally free sheaf is not relatively orientable over the Grassmannian.

\begin{tho}\label{THO - classe Euler del Pezzo degre 4}
With the description of lines in a del Pezzo surface of degree 4 (over a filed $\fF$ of cardinality at least 6) given in Proposition \ref{PROP - lines of degree 4 del Pezzo}, we get the "enriched count"
\[e^{CW}(\Symalg[2]{\cQ}\oplus \Symalg[2]{\cQ}) = 8h \in \Chowtilde[6]{\Grass(5,2),\cO}.\]
\end{tho}

\begin{proof}
It is a direct application of Theorem \ref{THO - classe Euler non orientable}. We only compute the Chern class degree. With some Schubert calculus, we find that $c_6(\Symalg[2]{\cQ}\oplus \Symalg[2]{\cQ}) = 16[pt]$ in $\Chowct[6]{G} \simeq \eZ.$
\end{proof}

\bibliographystyle{alpha}
\bibliography{biblio_memoire.bib}
%\printindex
\vspace{1cm}

\texttt{Institut de Mathématiques de Bourgogne, UMR 5584 CNRS, Université Bourgogne Europe,\\
F-21000, Dijon, France}

\textit{Email adress: victor.chachay@ube.fr}
\end{document}